\documentclass[a4paper,11pt]{amsart}

\usepackage{amsmath}
\usepackage{amssymb}
\usepackage{amsthm}
\usepackage{verbatim}
\usepackage[all]{xy}
\usepackage{enumerate}
\usepackage{comment}
\usepackage[colorlinks=true, citecolor=blue]{hyperref}
\usepackage{color}

\newtheorem{thm}{Theorem}[section]
\newtheorem{prop}[thm]{Proposition}
\newtheorem{cor}[thm]{Corollary}
\newtheorem{lem}[thm]{Lemma}

\theoremstyle{definition}

\newtheorem{rmk}[thm]{Remark}

\numberwithin{equation}{section}
\newcommand{\bT}{\mathbb{T}}
\newcommand{\bbT}{\mathsf{T}}
\newcommand{\C}{\mathbb{C}}
\newcommand{\E}{\mathbb{E}}
\newcommand{\Irr}{{\rm Irr}}
\newcommand{\QG}{\mathbb{G}}

\newcommand{\N}{\mathbf{N}}
\newcommand{\Pol}{{\rm Pol}}

\newcommand{\bt}{\mathbb{T}}
\newcommand{\Pweight}{\Lambda}

\newcommand{\supp}{{\rm supp}}
\newcommand{\bc}{\mathbb{C}}
\newcommand{\bz}{\mathbb{Z}}
\newcommand{\br}{\mathbb{R}}
\newcommand{\bn}{\mathbb{N}}

\title{Tracial central states on compact quantum groups} 

\author{Amaury Freslon}
\address{Universit\'e Paris-Saclay, CNRS, Laboratoire de Math\'ematiques d'Orsay, 91405 Orsay, France}
\email{amaury.freslon@universite-paris-saclay.fr}

\author{Adam Skalski}
\address{Institute of Mathematics of the Polish Academy of Sciences, ul.~\'Sniadeckich 8, 00--656 Warszawa, Poland}
\email{a.skalski@impan.pl}

\author{Simeng Wang}

\address{Institute for Advanced Study in Mathematics, Harbin Institute of Technology, Harbin 150001, China}

\email{simeng.wang@hit.edu.cn}

\begin{document}

\begin{abstract}
Motivated by classical investigation of conjugation invariant positive-definite functions on discrete groups, we study \emph{tracial central states} on universal C*-algebras associated with compact quantum groups, where centrality is understood in the sense of invariance under the adjoint action. We fully classify such states on $q$-deformations of compact Lie groups, on free orthogonal quantum groups, quantum permutation groups and on quantum hyperoctahedral groups.
\end{abstract}



\maketitle

\section{Introduction}

Consider a discrete group $\Gamma$. It is well-known that normalised  positive-definite functions $\varphi$ on $\Gamma$ correspond via the GNS representation to (cyclic) unitary representations $\pi_\varphi$ of $\Gamma$ and also to states $\omega_\varphi$ on  $C^*(\Gamma)$, the universal $C^*$-algebra of $\Gamma$, (or, equivalently, on the group algebra $\bc[\Gamma]$). If the group in question is non-abelian, the latter algebra is non-commutative, and it is natural to ask about the states which are \emph{tracial}, i.e.\ satisfy the condition $\omega_\varphi(x y) = \omega_\varphi(yx)$ for all $x, y \in C^*(\Gamma)$. Again, it is not difficult to see that the last property corresponds to the function $\varphi$ being \emph{conjugation invariant}, i.e.\ constant on conjugacy classes. Thus the convex weak*-closed set of all tracial states on $C^*(\Gamma)$ is naturally isomorphic to the set of conjugation invariant positive-definite functions on $\Gamma$. The study of the latter set  and its extremal points (sometimes described as the set of \emph{characters} of $\Gamma$) forms an important and active theme of geometric and combinatorial group theory, theory of representations and operator algebras (see for example \cite{BdH}, or the introduction to \cite{OSV}; related problems are also surveyed in \cite{CMP}). In particular recent years brought significant breakthroughs related to `character rigidity' of certain groups, which is understood as  admitting only very special extremal conjugation invariant functions (see \cite{PT} or \cite{BBHP}). 

In this article we initiate the study of an analogous question in the realm of discrete quantum groups, phrasing it in terms of their compact duals. Suppose that $\QG$ is a compact quantum group in the sense of \cite{Wor2}. We will be interested in tracial states on the associated Hopf $*$-algebra $\Pol(\QG)$; in the case of $\QG$ being the dual of a classical discrete group $\Gamma$ ($\QG= \widehat{\Gamma}$), these are precisely the objects introduced in the last paragraph. The general quantum setup offers however yet another feature: we can ask about \emph{central} tracial states of $\Pol(\QG)$, i.e.\ those which lie in the center of $\Pol(\QG)'$, where the latter space is equipped with the convolution product. In case where $\QG= \widehat{\Gamma}$ the centrality condition trivialises -- as the convolution product, corresponding to multiplication of positive-definite functions, is commutative -- but in general it provides a strong constraint on the class of traces we will analyse. Thus, the question studied in the paper is the following one:
\begin{itemize}
\item Given a compact quantum group $\QG$, can we describe explicitly (extremal) central tracial states on $\Pol(\QG)$?
\end{itemize}
Note that as the set of tracial central states on $\Pol(\QG)$ (equivalently, on $C^u(\QG)$, the universal $C^*$-completion of $\Pol(\QG)$) is a compact convex set inside a locally convex space, by the Krein-Milman theorem to understand its structure it is indeed sufficient to understand the extremal points. Whilst it might at first glance appear natural to study arbitrary tracial states, we will see that the class of central tracial states is much easier to classify, and at the same time already very important. Let us recall for example that central states on $\Pol(\QG)$ often encode key approximation properties of $\QG$ (\cite{Bra}, \cite{DCFY}) and play a fundamental role in the study of quantum L\'evy processes and their relationship to noncommutative geometry (\cite{CFK}).

The main results of this work are complete classifications of extremal tracial central states for a number of compact quantum groups, namely
\begin{itemize}
\item for $\QG_q$, i.e.\ $q$-deformations of a classical compact semisimple simply connected Lie group $G$ as constructed in \cite{KS98} -- extremal tracial central states are given by the points in the center of $G$;
\item for the free orthogonal group $O_N^+$ -- extremal tracial central states are the counit, the `alternating' character and the Haar state;
\item for the free permutation group $S_N^+$ -- extremal tracial central states are the counit and the Haar state;	
\item for the free hyperoctahedral group $H_N^+$ -- extremal tracial central states are the counit, the `alternating' character and the Haar state.
\end{itemize} 
As it turns out, in each of these cases  the set of central tracial states (and more generally, central tracial functionals on $\Pol(\QG)$) turns out to be rather small, containing only `obvious' elements and in particular finitely many extremal points. Moreover, it is a corollary of the methods that we use that all central tracial functionals extend continuously to the universal $C^*$-completion $C^u(\QG)$. The proofs of these facts turn out however to be highly non-trivial, and in each case require using different tools, from combinatorics of classical root systems of Lie algebras (\cite{Humphreys}), via Weingarten formula and calculus of partitions (\cite{Ban1}, \cite{BB1}, \cite{Fre}), to techniques from the theory of quantum convolution semigroups (\cite{Sch}, \cite{LiS}). We would like to note that although at first glance our conclusions appear similar to these appearing in the study of character rigidity (i.e.\ we see in our examples only very special central tracial states), the actual reasons seem conceptually different. In contrast to say \cite{PT} in the statements above we do not deal with quantum groups enjoying the geometric rigidity properties such as Kazhdan Property (T), but rather exploit the centrality property -- invisible for classical groups -- in conjunction with strong noncommutativity of the cases we study.

The plan of the paper is as follows: after this introduction in Section \ref{sec:prelim} we recall certain preliminaries, introduce central tracial states and their basic properties. In Section \ref{sec:deform} we first observe that tracial states always live on the Kac part of a given compact quantum group, and use this fact together with the combinatorial  arguments related to root systems to give a full description of central tracial states on $q$-deformations. Sections \ref{sec:orthogonal}, \ref{sec:permutation} and \ref{sec:hyperoctahedral} are devoted to characterising central tracial states on respectively $O_N^+$, $S_N^+$ and $H_N^+$; in each case the arguments use the Weingarten formula, but the second and the third  are much more involved. Finally in the Appendix we collect certain combinatorial computations concerning the Haar state of $S_N^+$, needed in Section \ref{sec:permutation}.

We will write $\bn_0$ for $\bn \cup \{0\}$.



\section{Preliminaries} \label{sec:prelim}

We will be working in the following setup: let $\QG$ be a compact quantum group in the sense of Woronowicz \cite{Wor2}, and let $\Pol(\QG)$ be the canonical Hopf $*$-algebra associated to $\QG$. We refer the reader for instance to \cite{NT} for the definitions of compact quantum groups and associated objects, as well as for proofs of the results from the general theory which will be used. The symbol $\Irr(\QG)$ will denote the set of (equivalence classes of) irreducible representations of $\QG$, and for $\alpha \in \Irr(\QG)$ we will denote by $U^{\alpha}$ a fixed representative, by $\chi_{\alpha}$ the associated character (the sum of diagonal elements of the $U^{\alpha}$, which does not depend on the choice of a representative) and by $d_{\alpha}$ the dimension of $\alpha$. We will always denote the trivial element of $\Irr(\QG)$ by $0$, so that $\chi_{0} = \mathbf{1}_{\Pol(\QG)}$ and $d_{0} = 1$. The unital $*$-algebra spanned by the characters of irreducible representations inside $\Pol(\QG)$ will be called the \emph{central subalgebra} or the \emph{subalgebra of class functions} and denoted $\Pol_{c}(\QG)$.

It is well-known that the Hopf $*$-algebras $\Pol(\QG)$ admit a characterisation as the so-called CQG-algebras \cite{DK}, that they admit universal $C^{*}$-completions (which we will denote by $C^{u}(\QG)$) and that there is a 1-1 correspondence between states (respectively, tracial states) on $\Pol(\QG)$ and  states (respectively, tracial states) on $C^{u}(\QG)$. Slightly abusing the language we will also speak simply  about states or tracial states on $\QG$.

The space of complex-valued functionals on $\Pol(\QG)$ is an algebra with respect to the natural convolution product: given $\phi,\psi: \Pol(\QG) \to \C$ we set $\phi \star \psi: = (\phi \otimes \psi)\Delta$. A functional $\phi$ on $\Pol(\QG)$ is called \emph{tracial} if for all $a,b \in \Pol(\QG)$ we have $\phi(ab) = \phi(ba)$.  We will be especially interested in \emph{central} functionals, i.e.\ those  $\phi : \Pol(\QG)\to \C$ for which there exists a family of complex numbers $(c_\alpha)_{\alpha\in \Irr(\QG)}$ such that for all $\alpha\in \Irr(\QG)$,
\begin{equation*}
\phi(U_{ij}^\alpha) = c_{\alpha}\delta_{ij}, \;\;\;i,j =1,\ldots, d_\alpha.
\end{equation*}
Note that each central functional is determined by the values it takes on the characters, with
\begin{equation*}
\phi(\chi_{\alpha}) = c_{\alpha}d_{\alpha} := \phi_{\alpha}.
\end{equation*}
It will therefore often be more convenient to describe a central functional via the sequence $(\phi_{\alpha})_{\alpha\in \Irr(\QG)}$. Being central has a natural interpretation in terms of the convolution product: it is easy to check that a functional $\phi: \Pol(\QG)\to \C$ is central in the above sense if and only if for any functional $\psi: \Pol(\QG)\to \C$ we have $\phi \star \psi = \psi \star \phi$. Note that a priori there is no guarantee that a central functional on $\Pol(\QG)$ admits a continuous extension to $C^{u}(\QG)$, see Remark \ref{rem:unbd}  (but for states this is the case, as explained in the last paragraph).

Every compact quantum group admits a \emph{counit}, which can be described as the central functional 
given by the formula
\begin{equation*}
\varepsilon(\chi_{\alpha}) = d_{\alpha}
\end{equation*}
for all $\alpha\in \Irr(\QG)$. The counit is in fact positive and multiplicative, so in particular is a central tracial state. Another distinguished case is that of the \emph{Haar state}, given by the formula $h(\chi_{0}) = 1$ and $h(\chi_{\alpha}) = 0$ for all $\alpha\in \Irr(\QG)\setminus\{0\}$. It is a central state, tracial if and only if $\QG$ is \emph{of Kac type}.

If $\QG$ is of Kac type, then it is well-known that we have a positive faithful, $h$-preserving conditional expectation
\begin{equation*}
\E:\Pol(\QG) \to \Pol_{c}(\QG).
\end{equation*}
For a given $\alpha\in \Irr(\QG)$, it follows from the Woronowicz-Peter-Weyl orthogonality relations that
\begin{equation*}
\mathbb{E}(x) = \sum_{\alpha\in\Irr(\QG)}h(\chi_{\alpha}^{*}x)\chi_{\alpha}, \;\;\; x\in \Pol(\QG).
\end{equation*}
Note that the sum above is in fact finite for any fixed $x\in \Pol(\QG)$. Note that the existence of $E$ with the properties above can be easily established using the general properties of the tracial von Neumann algebras, working at the level of the von Neumann algebra completion of $\Pol(\QG)$ with respect to the GNS representation of $h$.

This work focuses on states which are both central and tracial, and we will call them \emph{tracial central states} and sometimes abbreviate the term to TCS. It is obvious that the set of all TCS on a compact quantum group $\QG$, denoted by $\mathrm{TCS}(\QG)$, forms a weak$^{*}$-closed (in the weak$^{*}$-topology of $C^{u}(\QG)$) convex set. Thus, it is natural to look for extremal points of $\mathrm{TCS}(\QG)$.

\begin{rmk}
If $\QG = G$ is a classical compact group, the traciality condition is trivially satisfied, so that TCS are just central probability measures on $G$, i.e.\ those measures which are invariant under the adjoint action (so for example if $G$ is also abelian, simply all probability measures on $G$).

On the other hand if $\QG = \widehat{\Gamma}$ is the dual of a classical discrete group $\Gamma$, then tracial central states of $\QG$ are naturally identified with positive definite functions on $\Gamma$ which are also class functions (in other words, are conjugacy invariant). This time the centrality condition trivialises. Note that in this case the study of extremal points of the corresponding tracial states is an important topic in the representation theory. See for example \cite[Section 14]{BdH}, where the extremal points of the set $\mathrm{TCS}(\widehat{\Gamma})$ are introduced, called the \emph{Thoma dual} of $\Gamma$, denoted by $E(\Gamma)$ and studied in several examples. The structure of $E(\mathbb{F}_{2})$ was recently studied for example in \cite{OSV}; see also the survey \cite{CMP}.
\end{rmk}

We will later need to consider certain operations on tracial central functionals, in particular inspired by the theory of convolution semigroups of states (see \cite{Sch}). We gather here for convenience some elementary results concerning these.

\begin{lem}\label{lem:conv}
Suppose that $\phi, \psi: \Pol(\QG) \to \bc$ are functionals. Then 
\begin{itemize}
\item[(i)] If $\phi$ and $\psi$ are central, then so is  their convolution product $\phi \star \psi$; 
\item[(ii)] if $\phi$ and $\psi$ are tracial, then so is  their convolution product $\phi \star \psi$; 
\item[(iii)] the functionals $\exp_\star (t \phi):= \sum_{n=0}^\infty  \frac{t^n\phi^{\star n}}{n!}$ (where $\phi^{\star 0}:=\varepsilon$ and the series is convergent pointwise by the fundamental theorem on coalgebra) are tracial and central for each $t >0$ if and only if so is $\phi$;
\item[(iv)] if $\phi$ is a central state, $\alpha \in \Irr(\QG)$, then for any $t >0$ we have 
\[ \exp_\star(t(\phi - \varepsilon))(\chi_\alpha) = d_\alpha \exp(t \lambda_\alpha),\]
where $\lambda_\alpha= \frac{\phi(\chi_\alpha)}{d_\alpha}-1$. 
\end{itemize}
\end{lem}
\begin{proof}
Statements (i) and (ii) are an easy check. So is (iii), once we note that for every $a in \Pol(\QG)$ we have
\begin{equation*}
\phi(a)=\lim_{t \to 0^+} \frac{(\exp_\star (t \phi)) (a)- \varepsilon(a)}{t},
\end{equation*}
and both the properties we consider are preserved by pointwise limits. 	
 
Eventually, (iv) follows from a straightforward computation: fix $t>0$ and $\phi$ as above, and set $\widetilde{\phi} = \phi - \varepsilon$ (so that $\widetilde{\phi}$ is also central). Then for any $n\in \bn$ we have
\begin{align*} 
\widetilde{\phi}^{\star n}(\chi_{\alpha}) & = \sum_{i=1}^{d_{\alpha}}(\widetilde{\phi}\otimes\cdots\otimes \widetilde{\phi})(\Delta^{(n-1)}(U_{ii}^{\alpha})) = \sum_{i=1}^{d_{\alpha}}\sum_{i_{1}=1}^{d_{\alpha}} \cdots\sum_{i_{n-1}=1}^{d_{\alpha}}\widetilde{\phi}(U_{i, i_{1}}^{\alpha})\widetilde{\phi}(U_{i_{1}, i_{2}}^{\alpha})\cdots\widetilde{\phi} (U_{i_{n-1}, i}^{\alpha}) \\
& = \sum_{i=1}^{d_{\alpha}} (\widetilde{\phi} (U_{ii}^{\alpha}))^{n}  
 = d_{\alpha}\left(\frac{\phi(\chi_{\alpha})}{d_{\alpha}} - 1\right)^{n} = d_{\alpha}\lambda_{\alpha}^{n}.
\end{align*} 
Note that this formula works also for $n=0$. Thus indeed
\begin{equation*}
\exp_{\star}\left(t(\phi - \varepsilon)\right)(\chi_{\alpha}) = \sum_{n=0}^{\infty} \frac{t^{n} \widetilde{\phi}^{\star n}(\chi_{\alpha})}{n!} = \sum_{n=0}^{\infty} \frac{d_{\alpha} t^{n} \lambda_{\alpha}^{n}}{n!} = d_{\alpha}\exp(t \lambda_{\alpha}).
\end{equation*}
\end{proof}

\section{Deformations and the maximal Kac quantum subgroup} \label{sec:deform}

We start by investigating the case of $q$-deformations of semisimple simply connected compact Lie groups, denoted below $\QG_{q}$. These are well-known to be not of Kac type, so that the conditional expectation $\E$ onto the central subalgebra does not exist. However, we will first note that all tracial states factor through the largest quantum subgroup which is of Kac type, the so-called \emph{Kac part} introduced in \cite{Sol} (see also \cite[Section 3]{FFS}). This will then enable us to understand completely the structure of $\mathrm{TCS}(\QG_{q})$.

\begin{prop}\label{prop:Kactype}
Let $\QG$ be a compact quantum group with maximal Kac quantum subgroup $\QG_{\text{Kac}}$ (so that we have a surjective Hopf $*$-homomorphism $q_{\text{Kac}}:\Pol(\QG) \to \Pol(\QG_{\text{Kac}})$) and let $\phi : \Pol(\QG) \to \C$. Then $\phi$ is a \emph{tracial} state if and only if it is of the form  $\phi = \tau\circ q_{\text{Kac}}$ for some \emph{tracial} state $\tau$ on $\QG_{\text{Kac}}$.

In particular, if $\QG_{\text{Kac}}$ is a classical group $T$, then the tracial states on $\QG$ can be identified with probability measures on $T$.
\end{prop}

\begin{proof}
This is a consequence of the construction of the maximal Kac quotient in \cite[Appendix A]{Sol}. Note that So\l tan's construction is presented in the $C^{*}$-algebraic context, but it can be run in the framework of CQG-algebras, showing that if we define $\QG_{\text{Kac}}$ as the maximal quantum subgroup of $\QG$ of Kac type, as say considered in \cite[Section 3]{FFS}, then $\Pol(\QG_{\text{Kac}}) = \Pol(\QG)/J$, where
\begin{equation*}
J = \bigcap_{\tau\in\mathrm{Tr}(\QG)}\{b \in \Pol(\QG): \tau(b^*b) = 0\},
\end{equation*}
where $\mathrm{Tr}(\QG)$ denotes the set of all tracial states on $Pol(\QG)$.
This essentially ends the proof, showing that every tracial state on $\Pol(\QG)$ has to factor through $J$.
\end{proof}

\begin{rmk}
Note that this is not clear whether all tracial (even central) functionals on $\Pol(\QG)$ need to factor through $\Pol(\QG_{\text{Kac}})$.
\end{rmk}

The above result shows that for each  $q$-deformation of a classical semisimple simply connected compact Lie group $G$, tracial states on $\QG_{q}$ correspond to the probability measures on the maximal torus of $G$. We will now describe precisely these which are in addition central. We begin with $SU_{q}(2)$, where we can provide a direct argument using the explicit formul\ae{} for representations. Let $ q\in [-1,1] \setminus\{0\}$. Recall from \cite{Wor1} that $\Pol(SU_{q}(2))$ is generated by two elements $\alpha$ and $\gamma$ such that the matrix
\begin{equation*}
u:=\left(\begin{array}{cc}
\alpha & -q\gamma^{*} \\
\gamma & \alpha^{*}
\end{array}\right)
\end{equation*}
is unitary. The coproduct $\Delta : \Pol(SU_{q}(2))\to \Pol(SU_{q}(2))\otimes \Pol(SU_{q}(2))$ is given on the generators by
\begin{equation*}
\Delta(\alpha) = \alpha\otimes \alpha - q\gamma^{*}\otimes \gamma \quad \& \quad \Delta(\gamma) = \gamma\otimes \alpha + \alpha^{*}\otimes\gamma.
\end{equation*}
We will need the following description of the irreducible representations of $SU_{q}(2)$ established in \cite{Wor1}: they can be indexed by non-negative half-integers in such a way that $u^{0}$ is the trivial representation, $u^{1/2} = u$ and for all $l\in \frac{1}{2}\bn$,
\begin{equation*}
	u^{\frac{1}{2}}\otimes u^{l} = u^{l-\frac{1}{2}}\oplus u^{l+\frac{1}{2}}.
\end{equation*}

We are now ready to classify the TCS on $SU_{q}(2)$ for $q \notin \{-1,1\}$.

\begin{prop}
Let $q \in (-1,1)\setminus\{0\}$. Every tracial central state on $SU_{q}(2)$ is a convex combination of the counit and of the character $\varepsilon_{\mathrm{alt}}: \Pol(SU_{q}(2)) \to \C$, determined by the formula $\varepsilon_{\mathrm{alt}}(\alpha) = -1$, $\varepsilon_{\mathrm{alt}}(\gamma) = 0$. 
\end{prop}

\begin{proof}
Proposition \ref{prop:Kactype} implies that any central tracial state $\phi : \Pol(SU_{q}(2)) \to \bc$ is of the form $\phi = \tau_\mu \circ q_{\bT}$, where $\mu \in \textup{Prob}(\bT)$ and $q_{\bT} : \Pol(SU_{q}(2))\to \Pol(\bT))$ is the homomorphism given by $q_{\bT}(\alpha) = z, q_{\bT}(\gamma) = 0$ (the fact that $\bT$ is the Kac part of $SU_{q}(2)$ is well known, and formally stated in \cite[Lemma 4.10]{Tomatsu}). Let $l\in \frac{1}{2}\bn_0$ and let $u^{l}$ denote the $l$-th irreducible representation of $SU_{q}(2)$ (which is $2l+1$-dimensional). Then, the formul\ae{} \cite[Section 4, p.108]{Koo} imply that we have
\begin{equation*}
q_{\bt}^{(2l+1)}(u^l) = [z^{-2n}\delta_{n,m}]_{n,m=-l}^l,
\end{equation*}
where $q_{\bt}^{(2l+1)}(u^{l})$ denotes the suitable matrix lifting of $q_{\bt}$. Thus, if $\phi$ is a central functional, then we must have some coefficients $c_{l}\in \bc$ ($l\in \frac{1}{2}\bn_0$) with $c_{0} = 1$ such that for each $l\in \frac{1}{2}\bn_0$,
\begin{equation*}
\tau_\mu (z^{-2n}) = c_l, \;\;\; n=-l, -l+1, \ldots, l-1, l.
\end{equation*}
In other words, we must have some $c\in \bc$ such that for each $k\in \bn$ there is
\begin{equation*}
\tau_\mu (z^{2k})= 1, \;\; \tau_\mu(z^{2k+1}) = c.
\end{equation*}
It is easy to check that the character $\varepsilon_{\mathrm{alt}}$ corresponds to $c =-1$; then we can easily deduce that in fact $c\in [-1,1]$, which ends the proof (recall that the counit corresponds to $c=1$).
\end{proof}

Note that the direct counterpart of the above result cannot hold for $q \in\{-1,1\}$ simply because both $SU(2)$ and $SU_{-1}(2)$ are of Kac type, so that the tracial central states need not live on the torus subgroup (as the example of the Haar state shows).

To upgrade the last proposition to a statement valid for all $q$-deformations one needs to decode the formulas for the characters of $\Pol(\mathbb{G}_{q})$, identified with the points of the maximal torus of the underlying classical Lie group. Note that to make sense of the formulas below we can and do identify a maximal torus inside the classical group $G$ with the maximal classical subgroup inside $\QG_{q}$. We refer the reader to \cite{KS98} for the definition and properties of $\QG_{q}$.

\begin{prop}\label{prop:qdeformabstract}
Let $G$ be a simply connected compact semisimple Lie group, with a maximal torus $\bbT$, space of weights $\Pweight\cong \widehat{\bbT}$ and dominant weights $\Pweight^{+}$, and let $q \in (0,1)$. Given $\omega\in \Pweight^{+}$, we write $\Pi(\omega)$ for the saturated subset of $\Pweight$ associated to the highest weight $\omega$ (see for example \cite[Section 13.4]{Humphreys}). Then there is a one-to-one correspondence between
\begin{itemize}
\item[(i)] tracial central states $\tau$ on $\Pol(\mathbb{G}_q)$;
\item[(ii)] probability measures $\mu$ on $\bbT$ which satisfy the following: for every $\omega \in \Pweight^{+}$, there exists a constant $c_{\omega}\in \bc$ such that for every $\omega' \in \Pweight$, $\omega' \in \Pi(\omega)$, we have
\begin{equation*}
\int_{\bbT} \omega' d\mu = c_{\omega}.
\end{equation*}
\end{itemize}
The correspondence is given by the formula
\begin{equation}\label{formula:qTraces}
\tau(x) = \int_{\bbT} q_{\bbT}(x) d\mu, \;\; x \in \Pol(\mathbb{G}_q),
\end{equation}
where $q_{\bbT} : \Pol(\mathbb{G}_{q})\to \Pol(\bbT)$ is the surjective Hopf $*$-algebra map identifying $\bbT$ as a closed subgroup of $\mathbb{G}_{q}$.
\end{prop}

\begin{proof}
The fact that the maximal Kac quotient of $\mathbb{G}_{q}$ is $\bbT$ was proved in \cite[Lemma 4.10]{Tomatsu}. Hence, by Proposition \ref{prop:Kactype}, any tracial state on $C(\mathbb{G}_{q})$ (recall that $\mathbb{G}_{q}$ is coamenable, as follows for example from \cite{Tomatsu}, so that $C(\mathbb{G}_{q}) = C^{u}(\mathbb{G}_{q})$) is determined by a probability measure on $\bbT$, which we will denote by $\mu$. It is easy to see that conversely, for any $\mu\in \textup{Prob}(\bbT)$, the formula \eqref{formula:qTraces} defines a tracial state on $\Pol(\mathbb{G}_{q})$.

We are interested in central states; to that end we need to use the explicit formula for the map $q_{\bbT} :\Pol(\mathbb{G}_{q})\to \Pol(\bbT)$. Recall that we denote the set of weights associated with the Lie group $G$ by $\Pweight\cong \widehat{\bbT}$, with the set of dominant weights $\Pweight^{+}$. By results of Korogodskii-Soibelman in \cite{KS98} (see also the beginning of \cite[Subsection 4.1]{KrajczokSoltan}), each $\omega\in \Pweight^{+}$ determines an irreducible representation $U_{\omega}$ of $\mathbb{G}_{q}$ on a Hilbert space $H_{\omega}$, and in turn each irreducible representation of $\mathbb{G}_{q}$ is equivalent to one of the $U_{\omega}$'s, so that $\Irr(\mathbb{G}_{q})$ can be identified with $\Pweight^{+}$. Given a dominant weight $\omega\in \Pweight^{+}$ and a weight $\omega'\in \Pweight$, we have $\omega'\in \Pi(\omega)$ if and only if $H_{\omega}$ contains non-zero vectors of weight $\omega'$; we can then choose an orthonormal basis $(e_{1}, \cdots, e_{d_{\omega}})$ of $H_{\omega}$ such that each vector $e_{i}$ (for $i=1, \cdots, d_{\omega}$) has a well-defined weight $\omega_{i}\prec \omega$ (so that the set $\{\omega_{1}, \cdots, \omega_{d_{\omega}}\} = \Pi(\omega)$). Furthermore, the character (now understood simply as a multiplicative functional) associated with the element $t\in \bbT$ is given by the following formula:
\begin{equation*}
\chi_{t} (U^{\omega}_{i,j}) = \langle t, \omega_{i} \rangle \delta_{ij}. \;\;\; \omega\in \Pweight, i,j=1, \cdots, d_{\omega}.
\end{equation*}
	Suppose then that $\tau:C(\mathbb{G}_q)\to \bc$ is a tracial state, associated to the measure $\mu \in \textup{Prob}(\bbT)$.
	We have then for each $\omega \in \Pweight, i,j=1,\ldots, d_\omega$ the following equality:
	\[ \tau(U^{\omega}_{i,j}) = \int_{\bbT} \chi_t (U^{\omega}_{i,j}) d \mu(t) =  \int_{\bbT} \langle t, \omega_i \rangle \delta_{ij}d \mu(t).  \]
	Now it is easy to see that due to the description of irreducible representations given above, $\tau$ is central if and only if the integral above does not depend on $i=1,\ldots,d_{\omega}$; this in turn is equivalent to the condition stated in the proposition.
\end{proof}

The statement above would be fully satisfactory if we could describe explicitly probability measures on $\bbT$ satisfying the (purely classical/Lie-theoretic) condition in (ii) above. For that we need the following lemma, which uses the notation of Proposition \ref{prop:qdeformabstract}.

\begin{lem}\label{lem:equiv}
Assume that $\Phi$ is a root system (associated to a simply connected compact semisimple Lie group), with associated weight lattice $\Pweight$. Consider the following relation on $\Pweight$: 
\begin{equation*}
\omega' \sim \omega \text{ if } \omega'\in \Pweight,\, \omega\in \Pweight^{+},\, \omega'\in \Pi(\omega).
\end{equation*} 
Let $\approx$ be the equivalence relation generated on $\Pweight$ by $\sim$. Then, given $\lambda, \mu\in \Pweight$, we have $\lambda\approx \mu$ if and only if $\lambda - \mu \in \Pweight_{r}$ (the root lattice).
\end{lem}

\begin{proof}
The forward implication follows from \cite[Proposition 25.4]{Bump}, which implies among other things that if $\omega'\in \Pweight, \omega\in \Pweight^{+}, \omega'\in \Pi(\omega)$ then $\omega - \omega' \in \Pweight_{r}$. Assume therefore that $\lambda\in\Pweight$, and $\alpha\in \Delta$ is a simple root. A moment of thought shows that it suffices to prove that $\lambda \approx \lambda + \alpha$.
	
Start by assuming that $\lambda \in \Pweight^+$ and $\alpha \in \Delta$ is a simple root. Let $\nu\in \Pweight^+$ be a fundamental weight and let us write the set of simple roots as $\Delta = \{\alpha=\alpha_1, \ldots, \alpha_l\}$. By the arguments of \cite[Section 13.1]{Humphreys}, there exist integers $n\in \bn$, $k_{1}, \cdots, k_{l}\in \bn_0$ such that
\begin{equation*}
n\nu = \sum_{i=1}^{l}k_{i}\alpha_{i}.
\end{equation*}
It is clear that $\lambda\prec \lambda + n\nu$ (i.e.\ $\lambda + n\nu - \lambda$ is a sum of positive roots), so by \cite[Section 13.4, Lemma B]{Humphreys} we can conclude that $\lambda\in \Pi(\lambda + n\nu)$. We will also use the following somewhat stronger statement, which follows from the proof of \cite[Section 13.4, Lemma B]{Humphreys}: if only the coefficient $k_{1}$ is non-zero, then $\lambda + \alpha \in \Pi(\lambda + n\nu)$. Note that possibly changing $\nu$, we can achieve $k_{1} > 0$ (as the fundamental weights form a basis) -- so in particular we can deduce that $\lambda \approx \lambda + \alpha$. 
	
Assume now that $\mu\in \Pweight^{+}$ is another dominant weight and that $\lambda - \mu \in \Pweight_{r}$. This means that there exist some integer coefficients $k_{1}, \cdots, k_{r}\in \bz$ such that
\begin{equation*}
\lambda - \mu = \sum_{i=1}^{r} k_{i}\alpha_{i}.
\end{equation*}
Dividing the sum into parts corresponding to positive and negative coefficients we deduce that for some $p_{1},\cdots, p_{r}\in \bn_0$, we have that
\begin{equation*}
\lambda + \sum_{i=1}^{r} p_{i}\alpha_{i} = \mu + \sum_{i=1}^{r} p_{i}\alpha_{i}.
\end{equation*}
The previous paragraph then implies that $\lambda\approx \mu$.
	
Let now $\lambda\in \Pweight$ be arbitrary. By \cite[Section 13.2, Lemma A]{Humphreys}, there exists $w\in W$ (the Weyl group) such that $w\lambda \in \Pweight^{+}$. By the arguments above, we have that there exists $n\in \N$ and a fundamental weight $\mu\in \Pweight^{+}$ such that $w\lambda\in \Pi(w\lambda + n\mu)$. As the set on the right hand side is invariant under the Weyl group action (again by \cite[Section 13.4, Lemma A]{Humphreys}), for any $w'\in W$ we have that $w'\lambda \sim w\lambda + n\mu$. But this means that $w'\lambda \approx w\lambda$, so also $w'\lambda \approx \lambda$ for any $w'\in W$. Note that this in particular, by the very first part of the proof, implies that $w'\lambda - \lambda \in \Pweight_{r}$ for any $w'\in W$.
	
Considering again arbitrary $\lambda \in\Pweight$ and $\alpha\in \Delta$. By \cite[Section 13.2, Lemma A]{Humphreys}, we have $v,w \in W$ such that $w\lambda \in\Pweight^{+}$ and $v(\lambda + \alpha)\in \Pweight^{+}$. We then have
\begin{equation*}
w \lambda - v(\lambda + \alpha) = w\lambda - \lambda + \lambda - v\lambda - v\alpha \in \Pweight_{r},
\end{equation*}
so by the result of two paragraphs above $w\lambda \approx v(\lambda + \alpha)$. By the last paragraph, this implies that $\lambda \approx \lambda +\alpha$, concluding the proof.
\end{proof}

The next theorem is the main result of this section.

\begin{thm}\label{thm:qdeform}
Let $G$ be a simply connected compact semisimple Lie group and let $q\in (0,1)$. Recall that we can identify the center of $G$ with a closed subgroup of the maximal torus $\bbT$, so with a closed quantum subgroup of $\QG_{q}$. Then, there is a one-to-one correspondence between
\begin{itemize}
\item[(i)] tracial central states $\tau$ on $\Pol(\mathbb{G}_{q})$;
\item[(ii)] probability measures $\mu$ on $Z(G)$.
\end{itemize}
The correspondence is given by the formula
\begin{equation*}
\tau(x) = \int_{\bbT} q_{\bbT}(x) d\mu, \;\; x \in \Pol(\mathbb{G}_{q}).
\end{equation*}
In particular the space of extremal central tracial states on $\mathbb{G}_{q}$ is finite and can be identified with $Z(G)$.
\end{thm}

\begin{proof}
Fix  a probability measure $ \mu \in \textup{Prob}(\bbT)$. In view of Proposition \ref{prop:qdeformabstract} it suffices to show that $\mu$ satisfies the condition (ii) of  Proposition \ref{prop:qdeformabstract} if and only if it is supported on $Z(G)$. 
	
We identify again $\Pweight$ with the dual of $\bbT$; for clarity, given $\omega\in \Pweight$, we will write $\chi_{\omega}$ for the corresponding character on $\bbT$. Let $\Phi$ denote the root space of $G$. Lemma \ref{lem:equiv} implies that condition (ii) for $\mu$ is equivalent to the following fact: 
\begin{equation*}
\int_{\bbT} \chi_\omega d \mu = \int_{\bbT} \chi_{\omega+ \alpha}  d \mu = \int_{\bbT} \chi_{\omega}\chi_\alpha  d \mu, \;\;\; \omega \in \Pweight,\alpha \in \Phi.
\end{equation*}
By an easy Radon-Nikodym type argument together with the fact that a bounded measure is determined by its Fourier coefficients, we see that the displayed condition is equivalent to
\begin{equation*}
\chi_\alpha (t) =1, \;\;\; t \in \supp(\mu),\alpha \in \Phi.
\end{equation*}
In other words,
\begin{equation*}
\supp(\mu)\subset \bigcap_{\alpha \in \Phi} \textup{Ker}\, \chi_\alpha.
\end{equation*}
By \cite[Proposition  22.3 (i)]{Bump} the set on the right hand side equals $Z(G)$. This proves the first assertion. As for the finiteness of the set of extremal tracial central states, if follows from the fact that under the assumptions of the theorem, $Z(G)$ must be finite.
\end{proof}

\begin{rmk}
	We thank the referee for suggesting the following alternative argument proving Theorem \ref{thm:qdeform} without appealing to Lemma \ref{lem:equiv}. Note that the proof of Proposition \ref{prop:qdeformabstract} may be used to characterise these tracial central states on $\Pol(G)$ (i.e.\ conjugation invariant probability measures on $G$) which are supported on the maximal torus $\bbT\subset G$ as these which satisfy condition (ii) there. But the support of a conjugation invariant measure must be also conjugation invariant, and it is a simple consequence of Cartan's theorem on maximal tori (\cite[Theorem 16.5]{Bump}) that we have $\bigcap_{g \in G} g \bbT g^{-1} = Z(G)$. Whilst this argument is undoubtedly simpler, we believe that the statement of Lemma \ref{lem:equiv} might be of independent interest.
\end{rmk}

\section{Free orthogonal quantum groups}\label{sec:orthogonal}

We now turn to another family of examples, which in a sense generalizes  the case of $SU_{q}(2)$. These are the free orthogonal quantum groups first introduced in \cite{Wan1} (see also \cite{VDW}). Throughout this section, we fix an integer $N\geqslant 2$. The CQG-algebra of the quantum group $O_{N}^{+}$ is then the universal $*$-algebra generated by $N^{2}$ elements $(u_{ij})_{1\leqslant i, j\leqslant N}$ which are self-adjoint and satisfy the relations
\begin{equation*}
\sum_{k=1}^{N}u_{ik}u_{jk} = \delta_{ij} = \sum_{k=1}^{N}u_{ki}u_{kj}, \;\;\; i,j=1,\ldots,N.
\end{equation*}
The fusion rules of $O_{N}^{+}$ were computed in \cite{Ban1} and can be summarized as follows: the irreducible representations can be indexed by non-negative integers in such a way that $u^{0}$ is the trivial representation, $u^{1} = u$ and for any $n\in \bn$,
\begin{equation*}
u^{1}\otimes u^{n} = u^{n-1}\oplus u^{n+1}.
\end{equation*}
Note that it follows from this that the characters of irreducible representations are self-adjoint. Note also that $O_N^+$ is a compact quantum group of Kac type.

To describe $\mathrm{TCS}(O_{N}^{+})$, we first have to prove a technical result. The key idea is that the conjuction of traciality and centrality imposes strong constraints on the values of a state on characters. Precisely speaking, these values are completely determined by the values on the first two non-trivial characters. Proving this, however, requires some information on the Haar state when evaluated on polynomials in the generators. This can be done through the Weingarten formula obtained in \cite{BB1}. To state the formula, first recall that a \emph{pairing} of a set is a partition into subsets of cardinality two, and that if the underlying set is ordered, such a pairing is \emph{non-crossing} if it can be drawn with lines connecting pairs of points in a row without crossing each other. Fix for a moment $n \in \bn, n \geq 2$. Given a multi-index $i = (i_{1}, \cdots, i_{n})$ (with $i_j$ taking values in some fixed set) and a non-crossing pairing $\pi$ of $\{1, \cdots, n\}$, we will say that $i$ matches $\pi$ if whenever two elements $a$, $b$ are in the same pair of $\pi$, we have $i_{a} = i_{b}$. We will express this via a function $\delta_{\pi}$ defined on multi-indices as above, equal to $1$ for tuples matching $\pi$ and to $0$ otherwise. The formulation of the Weingarten formula involves a matrix indexed by the set of all non-crossing pairings of given size $n$, which we denote by $NC_{2}(n)$. Given $\pi, \sigma\in NC_{2}(n)$, we denote by $\pi\vee\sigma$ the partition obtained by merging any blocks of $\pi$ and $\sigma$ having a common point (note that this is not a pairing anymore). If $b(\pi)$ denotes the number of blocks of a partition, then the \emph{Gram matrix} of $O_{N}^{+}$ is given by the coefficients
\begin{equation*}
G_{N}(\pi, \sigma) = N^{b(\pi\vee\sigma)}, \;\; \pi, \sigma \in NC_2(n).
\end{equation*}
As soon as $N\geqslant 2$, the matrix $G_{N}$ is invertible, and its inverse is the Weingarten matrix $W_{N}$. The Weingarten formula then reads
\begin{equation}\label{eq:weingartenorthogonal}
h(u_{i_{1}j_{1}}\cdots u_{i_{n}j_{n}}) = \sum_{\pi, \sigma\in NC_{2}(n)}\delta_{\pi}(i)\delta_{\sigma}(j)W_{N}(\pi, \sigma),
\end{equation}
for all even $n\in \bn$ and  $ i_{1}, \ldots, i_{n},j_{1}, \ldots, j_n \in \{1, \ldots,N\}$. A first remarkable consequence of that formula is that moments are always real numbers. Let us now state and prove our technical result, recalling that for a central functional $\phi:\Pol(O_N^+)$ we write simply $\phi_n:=\phi(\chi_n)$ for the character of the representation $u^n$ (with $n \in \bn_0$).

\begin{prop}\label{prop:recursionorthogonal}
Let $N\geqslant 3$. There exist sequences $(a_{n})_{n\in \bn}$ and $(b_{n})_{n\in \bn}$ of real numbers such that for every tracial central linear functional $\phi$ on $O_{N}^{+}$ and every $n\in \bn$,
\begin{equation*}
\phi_{n} = \left\{\begin{array}{ccc}
a_{n}\phi_{1} & \textup{ if} & n \textup{ is odd, } \\
b_{n}\phi_{2} & \textup{ if}  & n \textup{  is even. } \\
\end{array}\right.
\end{equation*}
\end{prop}

\begin{proof}
For $1\leqslant i\leqslant N$, we will write $u_{i} = u_{ii}$ in order to lighten notations. The proof will be done by induction, setting $a_{1} = 1 = b_{2}$.

Let $n\in \bn$ and assume that the result holds for any $k\leqslant n$. We start with the case where $n = 2n'$ is even. Let us fix $1\leqslant i\neq j\leqslant N$ and set $A = (u_{i}u_{j})^{n'}$ and $B = u_{i}$. First, observe that by centrality,
\begin{align*}
\phi(AB) = \phi\circ\E(AB) = \sum_{i=0}^{+\infty}h(AB\chi_{i}^{*})\phi_{i}.
\end{align*}
By traciality, this equals $\phi(BA)$ so that we have
\begin{equation*}
\sum_{i=0}^{+\infty}h(AB\chi_{i}^{*})\phi_{i} = \sum_{i=0}^{+\infty}h(BA\chi_{i}^{*})\phi_{i}.
\end{equation*}
Recall that as mentioned earlier in Section \ref{sec:prelim}, the sums above are in fact finite.

We will now use the Weingarten formula to derive more information about the coefficients appearing in the equality above, starting with the terms involving $BA$. Consider, for $0\leqslant m\leqslant n+1$ and $1\leqslant k_{1}, \cdots, k_{m}\leqslant N$, the moment
\begin{equation*}
h((u_{i}u_{j})^{n'}u_{i}u_{k_{1}}\cdots u_{k_{m}}).
\end{equation*}
Observe that none of the first $n$ terms can be paired together via a non-crossing pairing, because between any two $u_{i}$'s there is an odd number of terms. Therefore, each of the first $n+1$ terms has to be paired with one of the last $m$ terms. If $m < n+1$, this is not possible, hence the sum vanishes in that case. Moreover, any non-crossing pairing contains an \emph{interval}, i.e.\ a set of the form $\{l,l+1,\ldots, l+k\}$, and by the preceding reasoning, if $m = n+1$ then that interval must pair the $(n+1)$-th term with the $(n+2)$-th one. Summing up, there is only one non-crossing pairing which matches the product, and it yields
\begin{align*}
h((u_{i}u_{j})^{n}u_{i}u_{k_{1}}\cdots u_{k_{n+1}}) & = h((u_{i}u_{j})^{n}u_{i}u_{i}(u_{j}u_{i})^{n}) \\
& = h(ABBA).
\end{align*}
Another way of stating what we have shown so far is
\begin{equation*}
h(AB\chi_{1}^{l}) = \sum_{k_{1}, \cdots, k_{l} = 1}^{N}h(ABu_{k_{1}}\cdots u_{k_{l}}) = \delta_{l, n+1}h(ABBA).
\end{equation*}
Now, it follows from the fusion rules that $\chi_{m}$ is a monic polynomial in $\chi_{1}$ of degree $m$, so that we have proven that in fact
\begin{equation*}
h(AB\chi_{m}) = \left\{\begin{array}{ccc}
0 & \text{ if } & m < n+1 \\
h(ABBA) & \text{ if } & m = n+1
\end{array}\right..
\end{equation*}

Turning now to $BA$, we still have that when applying the Weingarten formula to
\begin{equation*}
h(u_{i}^{2}(u_{j}u_{i})^{n'-1}u_j u_{k_{1}}\cdots u_{k_{m}}),
\end{equation*}
any of the terms in $(u_{j}u_{i})^{n'-1}$ must be paired with some $u_{k}$. The reason for that is that if some $u_{i}$ there was paired to one of the first two $u_{i}$'s, then this pairing would enclose an odd number of indices. As a consequence, the moment vanishes if $m < n-1$. It also vanishes for $m = n$ because there is then an odd number of terms in the moment. Therefore, we are left with $m = n-1$ and $m = n+1$. By the same argument as before, for $m = n-1$ there is only one possible non-crossing pairing, yielding
\begin{align*}
h(BA\chi_{n-1}) & = h(BA\chi_{1}^{n-1}) = \sum_{k_{1}, \cdots, k_{n-1} = 1}^{N}h(u_{i}^{2}(u_{j}u_{i})^{n'-1}u_j u_{k_{1}}\cdots u_{k_{n-1}}) \\
& = h(u_{i}^{2}(u_{j}u_{i})^{n'-1}(u_{j}u_{i})^{n'-1})  = h(B\tilde{A}\tilde{A}B),
\end{align*}
with $\tilde{A}= (u_{j}u_{i})^{n'-1}$.
Because $B\tilde{A}\tilde{A}B = B\tilde{A}(B\tilde{A})^{*}$ is positive, $B\tilde{A}$ is non-zero (as can be seen considering the classical group $O_N$) and $h$ is faithful on $\Pol(O_{N}^{+})$, we conclude that $h(BA\chi_{n-1})$ is non-zero.

Summing up, we have proven that
\begin{equation*}
[h(AB\chi_{n+1}) - h(BA\chi_{n+1})]\phi_{n+1} = h(BA\chi_{n-1})\phi_{n-1}
\end{equation*}
with $h(BA\chi_{n-1}) > 0$. Since this must hold for any tracial central state, it holds in particular for $\phi = \varepsilon$, for which $\phi_{n+1}, \phi_{n-1}\neq 0$. We conclude that $h(AB\chi_{n+1}) - h(BA\chi_{n+1})\neq 0$ and dividing shows that $\phi_{n+1} = a\phi_{n-1}$ for a certain $a\neq 0$. Applying the induction hypothesis then yields the result with $a_{n+1} = a  a_{n-1}$.

We still have to deal with the case of $n$ odd. This time we will use elements $A = (u_{i}u_{j})^{n'}$ and $B = u_{k}u_{i}$ for three distinct indices $i, j, k\in \{1, \ldots, N\}$. Note that it is in this case that we need the condition $N \geq 3$. Set $n=2n'+1$ and pick $A$, $B$ as above. Using the arguments as above we check first that $h(AB\chi_{m})=0$ unless $m=n+1$ and then that $h(BA\chi_m)=0$ unless $m \in \{n-1, n+1\}$. We claim further that $h(BA\chi_{n+1})=0$. This would give the desired statement, as then we'll obtain 
\[ h(AB \chi_{n+1})\phi_{n+1} = \phi(AB) = \phi(BA) = h(BA \chi_{n-1}) \phi_{n-1}, \]
and the conclusion holds as $h(AB \chi_{n+1}) = h(ABBA) >0$.

So indeed, consider $h(BA \chi_{n+1}) = h(u_k u_i^2 u_j (u_i u_j)^{n'-1} \chi_{n+1})$. By the arguments (and the notation) introduced below in the proof of Lemma \ref{lem:nonvanishing} it suffices to observe that the vectors of the form $e_k \otimes e_i^{\otimes 2} \otimes e_j \otimes (e_i \otimes e_j)^{\otimes n'}$ belong to the orthogonal complement of the projection onto the subspace $H_{n+1}$ inside $H_1^{\otimes (n+1)}$. The last statement can be however deduced from \cite[Lemma 6.2]{FTW}.
\end{proof}

Note that at this stage, Proposition \ref{prop:recursionorthogonal} only provides us with a necessary condition, and we formally do not know yet whether the formula above determines a tracial functional; we have only used the tracial condition on polynomials in diagonal elements. To clarify this and provide a complete statement, we first introduce another (apart from the counit) natural character on $O_{N}^{+}$.

\begin{prop}\label{prop:altcounit}
The formula
\begin{equation*}
\varepsilon_{alt}(\chi_{n}) = (-1)^{n} d_{n}
\end{equation*}
for $n\in \bn$ defines a central character on $\Pol(O_{N}^{+})$.	
\end{prop}

\begin{proof}
It suffices to note that as the diagonal matrix $-\mathrm{Id}\in M_{N}(\C)$ satisfies the defining relations of $\Pol(O_{N}^{+})$, there exists by universality a (unique) $*$-homomorphism $\varepsilon_{alt}:\Pol(O_{N}^{+})\to \C$ such that for all $1\leqslant i, j\leqslant N$, 
\begin{equation*}
\varepsilon_{alt}(u_{ij}) = -\delta_{ij}.
\end{equation*}
It is then easy to check (inductively) that this character is central.
\end{proof}

Before turning to the main result of this section, let us note a consequence of the conditions introduced in Proposition \ref{prop:recursionorthogonal}.

\begin{lem}\label{lem:linearcombination}
Every central functional on $\Pol(O_{N}^{+})$ satisfying the relations of Proposition \ref{prop:recursionorthogonal} (independently of the value taken at $1$) is tracial, and bounded with respect to the universal norm.
\end{lem}

\begin{proof}
It suffices to observe that such a functional $\phi$ can be decomposed as a linear combination of $\varepsilon$, $\varepsilon_{alt}$ and $h$. More precisely,
\begin{equation*}
\phi' = \phi + \frac{\phi_{1}}{d_{1}}\varepsilon_{\mathrm{alt}}
\end{equation*}
is a  tracial central functional vanishing on $\chi_{1}$, hence on $\chi_{2n+1}$ for all $n\in \bn$ by Proposition \ref{prop:recursionorthogonal}. We now need to kill also the even characters, and this can be done using the map $\varepsilon' = (\varepsilon + \varepsilon_{\text{alt}})/2$, which is a central functional vanishing on all odd characters. Setting
\begin{equation*}
\phi'' = \phi' - \frac{\phi'_{2}}{d_{2}}\varepsilon',
\end{equation*}
we have that $\phi''(\chi_{n}) = 0$ for all $n \in \bn$. Since $h$ vanishes on all characters except for $\chi_{0} = 1$, we conclude that $\phi'' - \phi''(1)h = 0$. Looking back at  the construction, we conclude that
\begin{align*}
\phi & = \phi' - \frac{\phi_{1}}{d_{1}}\varepsilon_{\mathrm{alt}}  = \phi'' + \frac{\phi'_{2}}{2d_{2}}\varepsilon + \frac{\phi'_{2}}{2d_{2}}\varepsilon_{\mathrm{alt}} - \frac{\phi_{1}}{d_{1}}\varepsilon_{\mathrm{alt}} \\
& = \phi''(1)h + \frac{\phi'_{2}}{2d_{2}}\varepsilon + \left(\frac{\phi'_{2}}{2d_{2}} - \frac{\phi_{1}}{d_{1}}\right)\varepsilon_{\mathrm{alt}}.
\end{align*}
\end{proof}

\begin{rmk}\label{rem:unbd}
Let us emphasize the fact that automatic boundedness of tracial central functional is specific to $O_{N}^{+}$ and very different from the classical case. For instance, if $G$ is an infinite compact group and $(\rho_{n})_{n\in \bn}$ a family of pair-wise non-equivalent irreducible representations (for example $G=\mathbb{T}$), then the functional sending $\chi_{\rho_{n}}$ to $n\dim(\rho_{n})$ is central on $\Pol(G)$ and tracial because the algebra is commutative. However, it does not extend to a bounded functional on $C(G)$.
\end{rmk}

We now have everything at hand to describe $\mathrm{TCS}(O_{N}^{+})$. We already have three extreme points -- namely $\varepsilon, \varepsilon_{alt}$ and $h$ -- and the next theorem tells us that there is nothing more.

\begin{thm}\label{thm:ON+}
For $N\geqslant 3$, any tracial central state on $O_{N}^{+}$ is a unique convex combination of $\varepsilon, \varepsilon_{alt}$ and $h$.
\end{thm}

\begin{proof}
In view of Lemma \ref{lem:linearcombination}, all we have to do is to prove that if $\phi$ is a state, then the coefficients appearing in the proof of Lemma \ref{lem:linearcombination} are positive and add up to $1$. First, let us write them explicitly:
\begin{equation*}
\phi = \left(1 - \frac{\phi_{2}}{d_{2}}\right)h + \frac{1}{2}\left(\frac{\phi_{1}}{d_{1}} + \frac{\phi_{2}}{d_{2}}\right)\varepsilon + \frac{1}{2}\left(\frac{\phi_{2}}{d_{2}} - \frac{\phi_{1}}{d_{1}}\right)\varepsilon_{\mathrm{alt}}.
\end{equation*}
It is clear that the sum is one, so that only positivity remains to be proven. For the first coefficient, this comes from the fact that for each $n \in \bn$ the norm of $\chi_{n}$ in $C^{u}(O_{N}^{+})$ is $d_{n}$, so that by the fact that states have norm $1$,
\begin{equation*}
\vert\phi_{2}\vert = \vert\phi(\chi_{2})\vert\leqslant \|\chi_{2}\| = d_{2}.
\end{equation*}

For the other two, let us set $\lambda = \phi_{1}/d_{1}$ and $\mu = \phi_{2}/d_{2}$ and observe that $\phi_{2n} = \mu d_{2n}$ for all $n\in \bn$. As a consequence, we have 
\begin{equation*}
0\leqslant \phi(\chi_{n}^{2}) = \sum_{k=0}^{n}\phi(\chi_{2k}) = 1 + \mu\left(\sum_{k=1}^{n}d_{2k}\right).
\end{equation*}
Since the sum in parenthesis diverges to $+\infty$, we must have $\mu \geqslant 0$. As for $\lambda$, consider the equality (again valid for all $n \in \bn$)
\begin{equation*}
\chi_{n+1}\chi_{n} = \sum_{k=0}^{n}\chi_{2k+1}.
\end{equation*}
Together with the fact that $\phi_{2n+1} = \lambda d_{2n+1}$, we get
\begin{equation*}
\phi(\chi_{n+1}\chi_{n}) = \lambda\left(\sum_{k=0}^{n}d_{2k+1}\right) = \lambda d_{n+1}d_{n}.
\end{equation*}
Next, we slightly reformulate our previous equality
\begin{equation*}
\phi(\chi_{n}^2) = 1 +\mu\left(\sum_{k=1}^{n}d_{2k}\right) = (1-\mu) + \mu\left(\sum_{k=0}^{n}d_{2k}\right) = (1-\mu) + \mu d_{n}^{2}.
\end{equation*}
We can now use the Cauchy-Schwarz inequality to obtain
\begin{align*}
\vert \lambda\vert^{2} & = \frac{1}{(d_{n+1}d_{n})^{2}}\vert\phi(\chi_{n+1}\chi_{n})\vert^{2}  \leqslant \frac{1}{(d_{n+1}d_{n})^{2}}\phi(\chi_{n}^{2})\phi(\chi_{n+1}^{2}) \\
& = \frac{1}{(d_{n+1}d_{n})^{2}}((1-\mu) + \mu d_{n}^{2})((1-\mu) + \mu d_{n+1}^{2}) 
= \left(\mu + \frac{1-\mu}{d_{n}^{2}}\right)\left(\mu + \frac{1-\mu}{d_{n+1}^{2}}\right).
\end{align*}
Letting $n$ go to infinity eventually yields $\vert\lambda\vert^{2}\leqslant \mu^{2}$, from which the positivity of $\mu + \lambda$ and $\mu - \lambda$ follows. The uniqueness of the decomposition is clear since the coefficients can be recovered by evaluating $\phi$ on the first three irreducible characters.
\end{proof}


\section{Quantum permutations}\label{sec:permutation}

The next family of examples which we will now study is the family of quantum permutation groups $S_{N}^{+}$, for $N\geqslant 6$. These are compact quantum groups introduced by Sh.\,Wang in \cite{Wan2} in the following way : $\Pol(S_{N}^{+})$ is the universal $*$-algebra generated by $N^{2}$ elements $(p_{ij})_{1\leqslant i, j\leqslant N}$ such that for all $1\leqslant i, j, l\leqslant N$,
\begin{itemize}
\item $p_{ij} = p_{ij}^{*}$;
\item $p_{ij}p_{il} = \delta_{jl}p_{ij}$ and $p_{ij}p_{lj} = \delta_{il}p_{ij}$;
\item $\displaystyle\sum_{i=1}^{k}p_{kj} = 1 = \displaystyle\sum_{j=1}^{k}p_{ik}$.
\end{itemize}
It can be endowed with a compact quantum group structure through the unique $*$-homomorphism $\Delta : \Pol(S_{N}^{+})\to \Pol(S_{N}^{+})\otimes \Pol(S_{N}^{+})$ such that
\begin{equation*}
\Delta(p_{ij}) = \sum_{k=1}^{N}p_{ik}\otimes p_{kj}, \;\;\; i,j =1, \ldots,N.
\end{equation*}

The quantum group $S_{N}^{+}$ is also of Kac type, but
let us now explain why this is different from the case of $O_{N}^{+}$ studied previously. According to \cite{Ban2}, the irreducible representations of $S_{N}^{+}$ can be indexed by non-negative integers in such a way that $u^{0}$ is the trivial representation, $p = (p_{ij})_{i,j=1}^N = u^{0}\oplus u^{1}$ is the fundamental representation and for any $n \in \bn$,
\begin{equation*}
u^{1}\otimes u^{n} = u^{n-1}\oplus u^{n}\oplus u^{n+1}.
\end{equation*}
We see right away that there is no parity preserving structure in the fusion rules, so that we cannot proceed as in Section \ref{sec:orthogonal}. Nevertheless, with extra work we will be able to prove that any tracial central state on $S_{N}^{+}$ is determined by its values on the characters $\chi_{1}$ and $\chi_{2}$. We will then use convolution semi-groups to show that in fact, $\mathrm{TCS}(S_{N}^{+})$ has only two extreme points: the Haar state and the counit.

\subsection{Reduction to two parameters}

Let us fix from now on a tracial central functional $\phi : \Pol(S_{N}^{+})\to \C$. Our first goal is to prove that $\phi$ is determined by its values on $\chi_{1}$ and $\chi_{2}$. Even though the basic strategy will be to use traciality as for $O_{N}^{+}$, the computations are subtler because the Weingarten formula (established in \cite{BB2}) is more involved for $S_{N}^{+}$ and difficult to use effectively. We will therefore have to go round this problem, and this starts with the following simple computation. From now on, we will write $p_{i} = p_{ii}$ to lighten the notation.

\begin{lem}\label{lem:nonvanishing}
Assume $N\geqslant 6$. Then, there exists three distinct indices $1\leqslant i, j, l\leqslant N$ such that for any $k\in \bn_0$,
\begin{align*}
h((p_{i}p_{j})^{k}p_{i}\chi_{2k+1}) & \neq 0 \\
h((p_{i}p_{j})^{k}p_{l}p_{i}\chi_{2k+2}) \neq 0
\end{align*}
\end{lem}

\begin{proof}
We start with the first case. Fix $k \in \bn_0$ and set $n = 2k+1$. Let $(e_{i})_{1\leqslant i\leqslant N}$ be the canonical basis of $\C^{N}$. Then, $(p_{i}p_{j})^{k}p_{i}$ is the coefficient of $u^{\otimes n}$ corresponding to the vector $\xi:=(e_{i}\otimes e_{j})^{\otimes k}\otimes e_{i}$, in the sense that
\begin{equation*}
(p_{i}p_{j})^{k}p_{i} = u^{\otimes n}_{\xi, \xi}:= (\omega_\xi \otimes \textup{id}_{\Pol(S_N^+)})(u^{\otimes n}),
\end{equation*}
and $\omega_\xi(T)=\langle \xi, T \xi \rangle$ for $T \in M_{N}^{\otimes n}$. Then $(p_{i}p_{j})^{k}p_{i}$ can be decomposed as a sum of coefficients of irreducible representations, and the corresponding coefficient of $u^{n}$ is given by the vector $P_{n}((e_{i}\otimes e_{j})^{\otimes k}\otimes e_{i})$, where $P_{n}$ denotes the orthogonal projection onto the subspace $H_{n}$ of $(\C^{N})^{\otimes n}$ corresponding to $u^{n}$. Now, observe that if $(v_{i})_{1\leqslant i\leqslant \dim(u^{n})}$ is an orthonormal basis of $H_{n}$, then for any $\xi, \eta\in H_{n}$ we have
\begin{align*}
h(u^{n}_{\xi, \eta}\chi_{n}) & = \sum_{i=1}^{\dim(u^{n})}h(u^{n}_{\xi, \eta}u^{n}_{v_{i}, v_{i}})  = \sum_{i=1}^{n}\frac{\langle \xi, v_{i}\rangle \overline{\langle \eta, v_{i}\rangle}}{\dim(u^{n})} = \frac{\langle\xi, \eta\rangle}{\dim(u^{n})}.
\end{align*}
As a consequence, we have
\begin{equation*}
h((p_{i}p_{j})^{k}p_{i}\chi_{n}) = \frac{\|P_{n}((e_{i}\otimes e_{j})^{\otimes k}\otimes e_{i})\|^{2}}{\dim(u^{n})}
\end{equation*}
and to prove that this is non-zero, it is enough to find one vector in $H_{n}$ which is not orthogonal to $(e_{i}\otimes e_{j})^{\otimes k}\otimes e_{i}$. To do this, recall that by \cite[Lem 6.2]{FTW}, for $N\geqslant 4$
\begin{equation*}
\zeta = \left((e_{2} - e_{1})\otimes (e_{4} - e_{3})\right)^{\otimes k}\otimes (e_{2} - e_{1}) \in H_{n}
\end{equation*}
because it is orthogonal to the range of $R_{\vert^{\odot n}}$ (we are using here the notations of \cite{FTW}), which is the orthogonal projection onto $H_{n}^{\perp}$. Taking $i = 2$ and $j = 4$, we get $\langle \zeta, (e_{i}\otimes e_{j})^{\otimes k}\otimes e_{i}\rangle = 1$ and the conclusion follows.

As for the second case, the reasoning is exactly the same, except that this time one needs to use the vector
\begin{equation*}
\zeta = (e_{2} - e_{1})\otimes (e_{3} - e_{2}))^{\otimes k}\otimes (e_{6} - e_{5})\otimes (e_{2} - e_{1}) \in H_{2k+2},
\end{equation*}
which makes sense as soon as $N\geqslant 6$. Taking $l = 6$ then yields the result.
\end{proof}

We can now simplify the description of $\mathrm{TCS}(S_{N}^{+})$. Recall that we write $\phi_{n}$ for $\phi(\chi_{n})$, $n \in \bn_0$.

\begin{prop}\label{prop:recursion}
Assume $N \geqslant 6$. For any integer $n\geqslant 3$, there exist $a_{n}, b_{n}\in \br$ such that for any tracial central functional $\phi$ on $S_{N}^{+}$, and any $n \in \bn$ we have
\begin{equation*}
\phi_{n} = a_{n}\phi_{1} + b_{n}\phi_{2}.
\end{equation*}
In particular, any tracial central functional is determined by its values on the first two non-trivial characters.
\end{prop}

\begin{proof}
We will prove the result by induction. We start with the case $n = 3$ and set $A = p_{i}p_{j}$ and $B = p_{i}$, where $i,j\in \{1, \ldots,N\}$ are given by Lemma \ref{lem:nonvanishing}. Then, for any tracial central functional $\phi$,
\begin{align*}
\phi(AB) & = h(p_{i}p_{j}p_{i}\chi_{0})\phi_{0} + h(p_{i}p_{j}p_{i}\chi_{1})\phi_{1} + h(p_{i}p_{j}p_{i}\chi_{2})\phi_{2} + h(p_{i}p_{j}p_{i}\chi_{3})\phi_{3}
\end{align*}
while using $p_{i}^{2} = p_{i}$ yields
\begin{align*}
\phi(BA) & = h(p_{i}p_{j}\chi_{0})\phi_{0} + h(p_{i}p_{j}\chi_{1})\phi_{1} + h(p_{i}p_{j}\chi_{2})\phi_{2}.
\end{align*}
Since $h(p_{i}p_{j}p_{i}\chi_{3}) \neq 0$ by Lemma \ref{lem:nonvanishing}, we get $\alpha, \beta, \gamma\in \C$ such that
\begin{equation*}
\phi_{3} = \alpha + \beta\phi_{1} + \gamma\phi_{2}.
\end{equation*}
Moreover, the coefficients above are explicitly given by sums of values of the Haar state on monomials. As in the case of $O_{N}^{+}$ mentioned above, the Weingarten formula for $S_{N}^{+}$ implies that these are all real numbers, hence $\alpha, \beta, \gamma\in \br$. Since this must hold for all $\phi$, and in particular for the Haar state $h$ which satisfies $h_{1} = h_{2} = h_{3} = 0$, we conclude that $\alpha = 0$.

Assume now that the result holds for some $n\in \bn$, $n \geq 3$. If $n = 2k$, then using $A = (p_{i}p_{j})^{k}$ and $B = p_{i}$, the same reasoning comparing $\phi(AB)$ and $\phi(BA)$ yields that $\phi_{2k+1}$ is a linear combination with real coefficients of the numbers $\phi_{n}$ for $n\leqslant 2k$, which together with the induction hypothesis yields the result. If instead $n = 2k+1$, the same argument works using $A = (p_{i}p_{j})^{k}$ and $B = p_{l}p_{i}$, where $l$ is given again by Lemma \ref{lem:nonvanishing}.
\end{proof}

\begin{rmk}\label{rmk:concreteSN}
One should stress the fact that the coefficients $a_{n}$ and $b_{n}$ need not be unique. Indeed, the conclusion of Theorem \ref{thm:classificationpermutations} states that the coefficients $a_{n}$, $b_{n}$ such that $\phi_{n} = a_{n}\phi_{1} + b_{n}\phi_{2}$ for all $n \in \bn$  are exactly those satisfying $a_{n}d_{1} + b_{n}d_{2} = d_{n}$. Note however that for the particular choices made in the proof of Proposition \ref{prop:recursion} we have
\begin{equation}
\label{a3}
a_3 = (h(p_ip_jp_i \chi_3))^{-1} (h(p_i p_j \chi_1) - h(p_ip_jp_i \chi_1)),
\end{equation}
\begin{equation}
\label{b3}
b_3 = (h(p_ip_jp_i \chi_3))^{-1} (h(p_i p_j \chi_2) - h(p_ip_jp_i \chi_2).
\end{equation}
\end{rmk}

\subsection{Complete classification}

In the case of $O_{N}^{+}$, we could conclude because we already had three tracial central states at hand, namely the Haar state $h$, the counit $\varepsilon$ and the signed counit $\varepsilon_{\text{alt}}$. But the last one is not defined in $S_{N}^{+}$, hence we only have two states to start with. This is not sufficient to conclude with Proposition \ref{prop:recursion}. To go further, we will need another tool, motivated by the theory of convolution semigroups and provided by Lemma \ref{lem:conv}. We will soon use this to conclude our description of $\mathrm{TCS}(S_{N}^{+})$, but we first isolate certain technical computations, which will allow us to deduce that the coefficients given by \eqref{a3}-\eqref{b3} are non-zero.

\begin{lem} \label{lem:computationa3}
Let $N\geqslant 4$, and let $i,j\in \{1, \ldots,N\}$, $i \neq j$. Then, $h(p_{i}p_{j}\chi_{1}) \neq h(p_{i}p_{j}p_{i}\chi_{1})$.
\end{lem}
\begin{proof}
 Setting
\begin{equation*}
\chi = \sum_{i=1}^{N}p_{i} = \chi_{1} + 1,
\end{equation*}
we have
\begin{align*}
h(p_{i}p_{j}\chi_{1}) - h(p_{i}p_{j}p_{i}\chi_{1}) & = h(p_{i}p_{j}(1-p_{i})\chi_{1}) 
= h(p_{i}p_{j}(1-p_{i})\chi) \\
& = h(p_{i}p_{j}\chi) - h(p_{i}p_{j}p_{i}\chi).
\end{align*}
Observe that $h$ is invariant under any permutation of the indices of the coefficients appearing above, as any permutation $\sigma \in S_N$ via a map $u_{kl} \mapsto u_{\sigma(k) \sigma(l)}$ induces an automorphism of the quantum group in question. It is easy to see that $\chi$ is fixed by any such automorphism. Hence,
\begin{align*}
h(p_{i}\chi^{2}) & = \sum_{k=1}^{N}h(p_{i}p_{k}\chi)  = h(p_{i}\chi) + (N-1)h(p_{i}p_{j}\chi).
\end{align*}
Using that argument again, we end up with
\begin{align*}
h(p_{i}p_{j}\chi) & = \frac{1}{N(N-1)}h(\chi^{3}) - \frac{1}{N(N-1)}h(\chi^{2}) = \frac{3}{N(N-1)},
\end{align*}
where we use  the fact that $h(\chi^2) = h(\chi_1^2 + 2 \chi_1 +1) = h(\chi_2 + 3 \chi_1+2)=2$, and a similar computation yielding $h(\chi^3)=5$.
Let us now compare this to the other term, which is
\begin{align*}
h(p_{i}p_{j}p_{i}\chi) & = \frac{1}{N-1}(h(p_{i}\chi p_{i}\chi) - h(p_{i}\chi)) = \frac{1}{N-1}h(p_{i}\chi p_{i}\chi) - \frac{1}{N(N-1)}h(\chi^{2}) \\
& = \frac{1}{N-1}h(p_{i}\chi p_{i}\chi) - \frac{2}{N(N-1)}.
\end{align*}
By the Cauchy-Schwarz inequality,
\begin{align*}
h(p_{i}\chi p_{i}\chi) & \leqslant h(p_{i}\chi\chi p_{i})^{1/2}h(\chi p_{i}p_{i}\chi)^{1/2} = h(p_{i}\chi^{2}) = \frac{1}{N}h(\chi^{3}) = \frac{5}{N}.
\end{align*}
Assume then that the statement we want to prove does not hold. Then $h(p_{i}p_{j}p_{i}\chi) = 3/N(N-1)$ and the Cauchy-Schwarz inequality is an equality. But this means that there is $\lambda > 0$ such that $p_{i}\chi = \lambda \chi p_{i}$. Taking norms yields $\lambda = 1$, but the relation $p_{i}\chi = \chi p_{i}$ does not hold in $\Pol(S_{N}^{+})$ with $N \geq 4$, as can be seen using for example the representation of the latter (universal) algebra induced by block matrices $\left(\begin{array}{cc} p & p^\perp \\ p^\perp & p \end{array}\right)$, $\left(\begin{array}{cc} q & q^\perp \\ q^\perp & q \end{array}\right)$, with two non-commuting projections $p$ and $q$. This yields a contradiction.
\end{proof}	
	
Unfortunately, the previous trick does not work for the other coefficient, and one has to resort to much more involved computations. Because these mainly rely on the properties of $S_{N}^{+}$ and do not involve tracial central states, we have gathered them in an appendix to this paper. This gives us the next corollary.

\begin{cor}\label{lem:nonzerocoefficients}
Let $N\geqslant 6$. We may assume that the conclusion of Proposition \ref{prop:recursion} holds with $a_{3}, b_{3}\neq 0$.
\end{cor}

\begin{proof}
We may use the explicit coefficients constructed in the proof of Proposition \ref{prop:recursion}, in particular with $a_3, b_3$ given by the formulas \eqref{a3}-\eqref{b3} (for suitably chosen $i,j\in \{1, \ldots,N\}$). Then the conclusion follows from Lemma \ref{lem:computationa3} and Theorem \ref{thm:computationb3}.	
\end{proof}

Once again, our final result is that apart from the obvious extremal tracial central states $\varepsilon$ and $h$, there is nothing else.

\begin{thm}\label{thm:classificationpermutations}
For $N\geqslant 6$, any tracial central state on $S_{N}^{+}$ is a convex combination of $\varepsilon$ and $h$.
\end{thm}

\begin{proof}
Let us fix a tracial central functional $\phi$ and for any $t \geqslant 0$ define $\phi_t = \exp_\star(t(\phi - \varepsilon))$.
By Lemma \ref{lem:conv} (iii) and Proposition \ref{prop:recursion}, we have for all $t\geqslant 0$
\begin{equation*}
\phi_{t}(\chi_{3}) = a_{3}\phi_{t}(\chi_{1}) + b_{3}\phi_{t}(\chi_{2}).
\end{equation*}
For each  $n \in \bn$ set $\lambda_{n} = \phi(\chi_{n})/d_{n} - 1$. Lemma \ref{lem:conv} (iv) shows then that  for all $t\geqslant 0$ we have
\begin{equation*}
d_{3}e^{t\lambda_{3}} = a_{3}d_{1}e^{t\lambda_{1}} + b_{3}d_{2}e^{t\lambda_{2}}.
\end{equation*}
Now, the functions $t\mapsto e^{at}$ are linearly independent for different values of $a$. Since $a_{3}, b_{3}\neq 0$ by Corollary \ref{lem:nonzerocoefficients}, we conclude that the exponents must be equal. In particular, $\lambda_{1} = \lambda_{2}$, which yields
\begin{equation*}
\phi_{2} = \frac{d_{2}}{d_{1}}\phi_{1}.
\end{equation*}
It then follows from Proposition \ref{prop:recursion} (and the properties of the counit)  that $\phi_{n} = d_{n}\phi_{1}/d_{1}$ for all $n\in \bn$.

Assume now that $\phi$ is a moreover a state. We will prove by contradiction that $\phi_{1}\geqslant 0$. By the Cauchy-Schwarz inequality, for each $n \in \bn$ we have
\begin{align*}
\vert \phi(\chi_{n})\vert^{2} & \leqslant \phi(\chi_{n}\chi_{n}^{*})  = \sum_{k=0}^{2n}\phi(\chi_{k}) = 1 + \sum_{k=1}^{2n}\frac{d_{k}}{d_{1}}\phi_{1}  = 1 + \left(\sum_{k=1}^{2n}\frac{d_{k}}{d_{1}}\right)\phi_{1}.
\end{align*}
As a consequence, if $\phi_{1} < 0$, then for $n$ large enough we have $\vert \phi_{n}\vert^{2} < 0$, which yields a contradiction. We can now conclude the proof. Indeed, since $h_{1} = 0$ and $\varepsilon_{1} = d_{1}\geqslant \phi_{1}$, we have $t = \phi_{1}/d_{1}\in [0, 1]$ and
\begin{equation*}
\phi_{1} = t\varepsilon_{1} + (1-t)h_{1}.
\end{equation*}
In other words, $\phi$ and  $t\varepsilon + (1-t)h$ are two tracial central states with the same value on $\chi_{1}$, hence they are equal by the first part of the proof.
\end{proof}

\begin{cor}\label{cor:SN}
Any tracial central linear functional on $\Pol(S_{N}^{+})$ is a linear combination of $\varepsilon$ and $h$, and therefore extends to a bounded functional on $C^{u}(S_{N}^{+})$.
\end{cor}

\begin{proof}	
It suffices to note that the first part of the proof of Theorem \ref{thm:classificationpermutations} does not use positivity of the tracial central functional $\phi$.
\end{proof}

\begin{rmk}
For $N = 2, 3$, $S_{N}^{+}$ coincides with the classical permutation group $S_{N}$, and therefore tracial central states can be identified with central probability measures on $S_N$ (i.e.\ with probability measures constant on conjugacy classes). For $N = 4, 5$, our methods are not sufficient, but we suspect that the statement of Theorem \ref{thm:classificationpermutations} is still true.
\end{rmk}

\begin{rmk}
	As noted by the referee, one can also replace the use of the convolution semigroups in the proof of Theorem \ref{thm:classificationpermutations} by considering finite convolution powers of  tracial central states. The same applies to the proof of Theorem \ref{thm:classifHyperoctahedral}.
\end{rmk}

\section{Quantum hyperoctahedral groups} \label{sec:hyperoctahedral}

In this section, we will extend the previous results to the hyperoctahedral quantum groups $H_{N}^{+}$, introduced in \cite{BBC}, once again for $N \geqslant 6$. The techniques are very similar to the ones used for $S_{N}^{+}$, but the difference in the structure of the representations theory implies some changes when evaluating the Haar state on monomials. Let us start with a few basics. Fix $N \in \bn$. The CQG-algebra $\Pol(H_{N}^{+})$ is the universal $*$-algebra generated by $N^{2}$ elements $(u_{ij})_{1\leqslant i, j\leqslant N}$ such that for all $1\leqslant i, j\leqslant N$,
\begin{itemize}
\item $u_{ij}^{*} = u_{ij}$;
\item $u_{ij}u_{ik} = \delta_{jk}u_{ij}^{2}$ and $u_{ij}u_{kj} = \delta_{ik}u_{ij}^{2}$;
\item the elements $(u_{ij}^{2})_{1\leqslant i, j\leqslant N}$ satisfy the defining relations of $\Pol(S_{N}^{+})$.
\end{itemize}
It can be endowed with a compact quantum group structure through the unique $*$-homomorphism $\Delta : \Pol(H_{N}^{+})\to \Pol(H_{N}^{+})\otimes \Pol(H_{N}^{+})$ such that
\begin{equation*}
\Delta(u_{ij}) = \sum_{k=1}^{N}u_{ik}\otimes u_{kj}, \;\;\;i,j=1,\ldots,N.
\end{equation*}
The representation theory of $H_{N}^{+}$, which is a compact quantum group of Kac type,  was computed in \cite{BV} and can be described as follows: let $W$ be the free monoid on the two-element set $\{0, 1\}$ and let $x\ast y$ denote the sum modulo $2$ for $x, y\in \{0, 1\}$. We define an operation, again denoted by $\ast$, on $W$ through the formula
\begin{equation*}
w\ast w' = w_{1}\cdots w_{n-1}(w_{n}\ast w'_{1})w'_{2}\cdots w'_{k}
\end{equation*}
as well as an involution $w\mapsto \overline{w}$ :
\begin{equation*}
\overline{w_{1}\cdots w_{n}} = w_{n}\cdots w_{1}.
\end{equation*}
Then, the irreducible representations of $H_{N}^{+}$ can be indexed by elements of $W$ in such a way that the empty word $\emptyset$ is the trivial representation, $u = u^{1}$, $(u_{ij}^{2})_{1\leqslant i, j\leqslant N} = u^{0}\oplus u^{\emptyset}$ and for any words $w, w'\in W$,
\begin{equation*}
u^{w}\otimes u^{w'} = \sum_{w = az, w' = \overline{z}b} u^{ab}\oplus u^{a\ast b}.
\end{equation*}
Note that for the empty set we interpret the rules above as $\emptyset \emptyset = \emptyset$, and $\emptyset \star \emptyset$ is simply omitted. Thus for example $u^{1}\otimes u^1 = u^{11}\oplus u^{0}\oplus u^{\emptyset}$.

As for the case of $S_{N}^{+}$, we will use traciality in combination with the non-vanishing of some values of the Haar state to deduce relations between the values of all tracial central states on characters. We start with the non-vanishing part. To lighten notations, for each $i =1, \ldots, N$ we will write $u_{i}$ for $u_{ii}$. Besides, we set $u_{i}^{0} = u_{i}^{2}$, which is justified by the fact that $u_{i}^{2k+1} = u_{i}$ and $u_{i}^{2k+2} = u_{i}^{2}$ for all $1\leqslant i\leqslant N$ and $k\in \bn$. 


Given a central functional  $\phi$ on $\Pol(H_{N}^{+})$ we will write as usual $\phi_{w}:= \phi(\chi_w)$ for $w\in W$. 

\begin{rmk} \label{rem:SNinside}
We can realise $\Pol(S_{N}^{+})$ inside $\Pol(H_{N}^{+})$ as the subalgebra generated by $u_{ij}^{2}$, $i,j = 1,\cdots, N$. In this picture, the representation $u^{0}\in \Irr_{H_{N}^{+}}$ corresponds to the generating irreducible representation of $S_{N}^{+}$. Moreover, every tracial central functional $\phi$ on $\Pol(H_{N}^{+})$ remains by definition central when restricted to $\Pol(S_{N}^{+})$(and obviously also tracial). Thus, we can use Corollary \ref{cor:SN} to deduce that for example $\phi_{00} = d_{2}/d_{1} \phi_0$, where $d_{i}$ denotes the dimension of the $i$-th irreducible representation of $S_{N}^{+}$. We refer the reader to \cite[Prop 3.2]{Lem} for details.
\end{rmk}

As before, we first need a non-vanishing result for specific values of the Haar state involving irreducible characters.

\begin{lem}\label{lem:nonvanishinghyperoctahedral}
Let $n \in \bn$, let $w = w_{1}\cdots w_{n}\in W$ and let $i, j, l\in \{1, \ldots, N\}$ be three distinct indices. If $n$ is odd, so that $n = 2k+1$ for some $k \in \bn_0$, then
\begin{equation*}
h\left(u_{i}^{w_{1}}u_{j}^{w_{2}}\cdots u_{i}^{w_{2k-1}}u_{j}^{w_{2k}}u_{i}^{w_{2k+1}}\chi_{w}^{*}\right) \neq 0.
\end{equation*}
If $n = 2k+2$ for some $k \in \bn_0$, then
\begin{equation*}
h\left(u_{i}^{w_{1}}u_{j}^{w_{2}}\cdots u_{i}^{w_{2k-1}}u_{j}^{w_{2k}}u_{l}^{w_{2k+1}}u_{i}^{w_{2k+2}}\chi_{w}^{*}\right) \neq 0.
\end{equation*}
\end{lem}

\begin{proof}
The reasoning is the same as for the proof of Lemma \ref{lem:nonvanishing}. Using the notation of the proof of \cite[Theorem 6.7(2)]{FTW} it is sufficient to prove that the vector
\begin{equation*}
\xi = e_{i, w_{1}}\otimes e_{j,w_{2}}\otimes \cdots\otimes e_{j,w_{2k}}\otimes e_{i, w_{2k+1}}
\end{equation*}
is not orthogonal to the range $H_w$ of the projection $P_{w}$ onto the carrier space of the irreducible representation $u^{w}$. Using the proof of \cite[Thm 6.7(2)]{FTW}, one sees that (with the notations therein)
\begin{equation*}
\zeta = v_{i, w_{1}}\otimes v_{j, w_{2}}\otimes \cdots\otimes v_{j, w_{2k}}\otimes v_{i, w_{2k+1}}\in H_{w}
\end{equation*}
and that $\langle \xi, \zeta\rangle \neq 0$. The proof for the even case is similar.
\end{proof}

The last lemma enables us to reduce the problem to a few initial values, namely those given by words of length at most $2$. However, we can even do better. To state this, let us denote, for $w\in W$, by $c(w)$ the sum of its letters modulo $2$.

\begin{prop}\label{prop:recursionhyperoctahedral}
Let $N\geqslant 6$. For any word $w\in W$ there exist coefficients $a_{w}, b_{w}\in \br$ such that for every tracial central functional $\phi: \Pol(H_N^+) \to \bc$  if $c(w) = 0$, then
\begin{equation*}
\phi_{w} = a_{w}\phi_{11} + b_{w}\phi_{0},
\end{equation*}
and if $c(w) = 1$, then
\begin{equation*}
\phi_{w} = a_{w}\phi_{10} + b_{w}\phi_{1}.
\end{equation*}
\end{prop}

\begin{proof}
We will prove the result by induction on the length of $w$. First note that it holds for words of length at most two. Indeed, in the odd case it suffices to see that by traciality of $\phi$, we have
\begin{align*}
	\phi_{01} & = \phi(\chi_{0}\chi_{1}) - \phi(\chi_{1}) = \phi(\chi_{1}\chi_{0}) - \phi(\chi_{1}) = \phi_{10},
\end{align*}
and in the even case we first note that for any non-empty word $w$ of length at most $2$ with $c(w)=0$ we have $a_{w}, b_{w}, c_{w}\in \br$ 
such that
\begin{equation*}
	\phi_{w} = a_{w}\phi_{11} + b_{w}\phi_{0} + c_{w}\phi_{00}.
\end{equation*}
By Remark \ref{rem:SNinside} we have $\phi_{00} = d_{2}/d_{1}\phi_{0}$, hence the result.

Fix now $n\in \bn$, $n\geqslant 2$ and assume that the result holds for words of length at most $n$. If $n = 2k$ is even, consider a word $w = w_{1}\cdots w_{2k+1}$, choose $i,j\in \{1, \ldots,N\}$, $ i \neq j$, and let $x = u_{i}^{w_{1}}u_{j}^{w_{2}}\cdots u_{j}^{w_{2k}}u_{i}^{w_{2k+1}}$ be as in Lemma \ref{lem:nonvanishinghyperoctahedral}. Then, $\phi(x)$ is a sum of multiples of $\phi_{w'}$ for words $w'$ of length strictly less than $2k+1$ with $c(w') = c(w)$, plus a non-zero multiple of $\phi_{w}$. On the other hand, by traciality we have
\begin{equation*}
\phi(x) = \phi\left(u_{i}^{w_{2k+1}+w_{1}}u_{j}^{w_{2}}\cdots u_{i}^{w_{2k-1}}u_{j}^{w_{2k}}\right).
\end{equation*}
This is a coefficients of a tensor product of $2k$ irreducible representations, hence is a sum of multiples of numbers $\phi_{w'}$ for words $w'$ of length at most $2k$ with $c(w') = c(w)$. Comparing the two expressions yields the result for $w$. If $n = 2k+1$, a similar argument can be applied using the second part of Lemma \ref{lem:nonvanishinghyperoctahedral} instead, yielding the result.
\end{proof}

As before, we will need at some point to use the fact that certain coefficients constructed in the above result  are non-vanishing. This is the content of the next lemma. We will use again a crucial property of the Weingarten formula for $H_{N}^{+}$. The only important thing is that a formula analogous to \eqref{eq:weingartenorthogonal} holds with the set $NC_{2}$ of non-crossing pair partitions replaced by the set of non-crossing partitions with all blocks of even size (see for instance \cite{BS}).

\begin{lem} \label{lem:nonzeroHN}
Let $N\geqslant 6$. We may assume that the coefficients $(a_{w}, b_{w})_{w \in W}$ satisfying the conclusions of Proposition \ref{prop:recursionhyperoctahedral} are such that $a_{111}, b_{111} \neq 0$.
\end{lem}

\begin{proof}	
Let $i,j\in \{1, \ldots, N\}$, $i \neq j$. Suppose that $\phi$ is a tracial central functional on $\Pol(H_N^+)$. Consider the element $u_i u_j u_i$ used in the proof of Proposition \ref{prop:recursionhyperoctahedral} for $w=111$. Note that we can view $u_i u_j u_i$ as a coefficient of the representation $u^{11} \otimes u^1 \approx u^{111} \oplus u^{10} \oplus u^{1}$.  Hence
\[ \phi (u_i u_j u_i) = h(u_i u_j u_i \chi_{111}^*) \phi_{111} +  h(u_i u_j u_i \chi_{10}^*) \phi_{10} +  h(u_i u_j u_i \chi_{1}^*)  \phi_{1}. \]
As on the other hand we can view 	$u_i u_j u_i$ as a coefficient of the representation $u^{1} \otimes u^{11} \approx u^{111} \oplus u^{01} \oplus u^{1}$, the second factor in the above sum vanishes -- simply as the consequence of the fact that by Woronowicz-Peter-Weyl relations any coefficient of  $u^{111} \oplus u^{01} \oplus u^{1}$ must be orthogonal to $\chi_{10}$. Furthermore, the Weingarten formula implies that the third factor vanishes as well, and we are left with
\[ \phi (u_i u_j u_i) = h(u_i u_j u_i \chi_{111}^*) \phi_{111} .\]
On the other hand $ u_ju_i^2 $ can be viewed as a coefficient of $   u^1 \otimes(u^0 \oplus u^{\emptyset}) \approx u^{10} \oplus u^1 \oplus u^1$, so that 
\begin{equation*}
\phi(u_{j}u_{i}^{2}) = h(u_{j}u_{i}^{2}\chi_{10}^{*})\phi_{10} +  h(u_{j}u_{i}^{2}\chi_{1}^{*})\phi_{1}.
\end{equation*}
Comparison of the last two displayed formulas, and the use of Lemma \ref{lem:nonvanishinghyperoctahedral} shows that it suffices to argue that 
\begin{equation*}
h(u_{j}u_{i}^{2}\chi_{10}^{*}) \neq 0, \;\;\; h(u_{j}u_{i}^{2}\chi_{1}^{*}) \neq 0.
\end{equation*}
The first statement is a direct consequence of Lemma \ref{lem:nonvanishinghyperoctahedral}, while the  Weingarten formula yields $h(u_{j}u_{i}^{2}\chi_{1}^{*}) =  h(u_{j}^{2}u_{i}^{2})\neq 0$ (the last inequality follows from faithfulness of the Haar state on $\Pol(H_{N}^{+})$).
\end{proof}

We are now almost ready to classify all tracial central functionals on $H_{N}^{+}$. Note that we already have three extremal ones at hand, namely the Haar state $h$, the counit $\varepsilon$ and the signed counit $\varepsilon_{\text{alt}}$ (observe that the coefficients of $-\mathrm{Id}\in M_{N}(\C)$ satisfy the defining relations of $\Pol(H_{N}^{+})$). Our main result in this section is that these are the only ones. To prove it, we will need to improve Proposition \ref{prop:recursionhyperoctahedral} to show that $\phi_{w}$ in fact only depends on $\phi_{c(w)}$. The even case turns out to be subtle, hence we establish the relevant statement in a separate lemma.

\begin{lem}\label{lem:evencase}
Let $N\geqslant 6$. Then, for any tracial central functional $\phi$ on $H_{N}^{+}$, $\phi_{11} = N\phi_{0}$.
\end{lem}

\begin{proof}
The proof will proceed in three steps. The first one consists in relating $\phi_{11}$ with $\phi_{0}$ with the help of $\phi_{101}$ and $\phi_{110}$. To do this, observe that
\begin{equation*}
\chi_{1}\chi_{0}\chi_{1} = \chi_{101} + 2\chi_{11} + \chi_{0} + 1
\end{equation*}
and
\begin{equation*}
\chi_{1}\chi_{1}\chi_{0} = \chi_{110} + \chi_{00} + 2\chi_{0} + \chi_{11} + 1.
\end{equation*}
Applying $\phi$ and using traciality, this leads to
\begin{equation}\label{eq:relationlengththree}
\phi_{101} - \phi_{0} - \phi_{00} = \phi_{110} - \phi_{11}.
\end{equation}

The second step is now to show that $\phi_{101}$ is in fact a combination of $\phi_{0}$ and $\phi_{00}$. Let us consider the element $u_{i}u_{j}^{2}u_{i}$, $i,j\in \{1, \ldots,N\}$, $ i \neq j$. It is a coefficient of $u^{1}\otimes u^{0}\otimes u^{1}$, hence
\begin{equation*}
\phi(u_{i}u_{j}^{2}u_{i}) = h(u_{i}u_{j}^{2}u_{i}\chi_{101}^{*})\phi_{101} + h(u_{i}u_{j}^{2}u_{i}\chi_{11}^{*})\phi_{11} + h(u_{i}u_{j}^{2}u_{i}\chi_{0}^{*})\phi_{0} + h(u_{i}u_{j}^{2}u_{i}).
\end{equation*}
Now, observe that
\begin{equation*}
\sum_{k, l = 1}^{N}u_{k}u_{l} = \chi_{1}\chi_{1} = \chi_{11} + \chi_{0} + \chi_{\emptyset}.
\end{equation*}
Since the sum of the last two characters is the sum of the coefficients $u_{k}^{2}$, we see that $\chi_{11}$ is the sum of $u_{k}u_{l}$ for $k,l\in \{1, \ldots,N\}$, $k\neq l$, hence
\begin{align*}
h(u_{i}u_{j}^{2}u_{i}\chi_{11}^{*}) & = \sum_{k\neq l}h(u_{i}u_{j}^{2}u_{i}u_{k}u_{l}) = 0
\end{align*}
by the Weingarten formula (remember this only involves non-crossing partitions with blocks of even size). In other words, the  term corresponding to $\phi_{11}$ in the equality above disappears. Moreover, by traciality, $\phi(u_{i}u_{j}^{2}u_{i})$ equals
\begin{equation*}
\phi(u_{i}^{2}u_{j}^{2}) = h(u_{i}^{2}u_{j}^{2}\chi_{00}^{*})\phi_{00} + h(u_{i}^{2}u_{j}^{2}\chi_{0}^{*})\phi_{0} + h(u_{i}^{2}u_{j}^{2}).
\end{equation*}
Combining the two equalities, we find $\alpha, \beta\in \br$ such that
\begin{equation*}
h(u_{i}u_{j}^{2}u_{i}\chi_{101}^{*})\phi_{101} = \alpha\phi_{00} + \beta\phi_{0}.
\end{equation*}
By Lemma \ref{lem:nonvanishinghyperoctahedral}, $h(u_{i}u_{j}^{2}u_{i}\chi_{101}^{*})\neq 0$ so that we can divide by it. Moreover, by Remark \ref{rem:SNinside}, we know that $\phi_{00}$ is a fixed multiple of $\phi_{0}$. As a conclusion, $\phi_{101}$ is a multiple of $\phi_{0}$. Since the multiplicative factor does not depend on $\phi$, we can compute it using the counit to conclude that
\begin{equation*}
\phi_{101} = \frac{\varepsilon(\chi_{101})}{\varepsilon(\chi_{0})}\phi_{0} = \frac{N(N^{2} - 3N + 1)}{N-1}\phi_{0},
\end{equation*}
where we use \cite[Theorem 9.3]{BV}.

The third step is similar to the second one, since we now show that $\phi_{110}$ is a multiple of $\phi_{11}$. This is done by using the monomial $u_{i}u_{j}u_{i}^{2}$, again for $i,j \in\{1, \ldots,N\}$, $i \neq j$.. This time, we have
\begin{equation*}
\phi(u_{i}u_{j}u_{i}^{2}) = h(u_{i}u_{j}u_{i}^{2}\chi_{110}^{*})\phi_{110} + h(u_{i}u_{j}u_{i}^{2}\chi_{11}^{*})\phi_{11} + h(u_{i}u_{j}u_{i}^{2}\chi_{00}^{*})\phi_{00} + h(u_{i}u_{j}u_{i}^{2}\chi_{0}^{*})\phi_{0} + h(u_{i}u_{j}u_{i}^{2}).
\end{equation*}
We first observe that
$h(u_{i}u_{j}u_{i}^{2}) = 0$ by the Weingarten formula,
\begin{align*}
h(u_{i}u_{j}u_{i}^{2}\chi_{0}^{*}) & = h(u_{i}u_{j}u_{i}^{2}(\chi_{0} + 1)^{*}) = \sum_{k=1}^{N}h(u_{i}u_{j}u_{i}^{2}u_{k}^{2}) = 0
\end{align*}
and
\begin{align*}
h(u_{i}u_{j}u_{i}^{2}\chi_{00}^{*}) & = h(u_{i}u_{j}u_{i}^{2}(\chi_{00} + \chi_{0} + 1)^{*}) = \sum_{k, l=1}^{N}h(u_{i}u_{j}u_{i}^{2}u_{k}^{2}u_{l}^{2}) = 0.
\end{align*}
Using traciality and the fact that $u_{i}^{3} = u_{i}$, we also have
\begin{align*}
\phi(u_{i}u_{j}u_{i}^{2}) & = \phi(u_{i}u_{j}) = h(u_{i}u_{j}\chi_{11}^{*})\phi_{11}.
\end{align*}
It readily follows from the Weingarten formula that $h(u_{i}u_{j}\chi_{11}^{*}) = h(u_{i}u_{j}u_{j}u_{i})$, so that comparing the two expressions leads to
\begin{equation*}
h(u_{i}u_{j}u_{i}^{2}\chi_{110}^{*})\phi_{110} = [h(u_{i}u_{j}u_{j}u_{i}) - h(u_{i}u_{j}u_{i}^{2}u_{j}u_{i})]\phi_{11}.
\end{equation*}
By Lemma \ref{lem:nonvanishinghyperoctahedral}, $h(u_{i}u_{j}u_{i}^{2}\chi_{110}^{*})\neq 0$ and we can divide to get $\phi_{110} = \delta\phi_{11}$ for some $\delta \in \br$. Once again, applying this equality to the counit yields $\delta = N-2$. Summing up, Equation \eqref{eq:relationlengththree} now reads $\gamma\phi_{0} = (\delta-1)\phi_{11}$ and since $\delta\neq 1$, we have that  $\phi_{11}$  is indeed proportional to $\phi_{0}$. Using the counit we find that the proportionality constant is $d_{11}/d_{0} = N$.
\end{proof}

We are now ready for the complete classification.

\begin{thm}\label{thm:classifHyperoctahedral}
For $N\geqslant 6$, any tracial central state on $H_{N}^{+}$ is a convex combination of $h$, $\varepsilon$ and $\varepsilon_{\text{alt}}$.
\end{thm}

\begin{proof}
As explained above, we need to show that for each $w \in W$ the value $\phi_{w}$ is a fixed multiple of $\phi_{c(w)}$. For $c(w) = 0$ this follows already from Proposition \ref{prop:recursionhyperoctahedral} and Lemma \ref{lem:evencase}. 

For $c(w) = 1$, we can use the ``semi-group trick'' as in the proof of Theorem \ref{thm:classificationpermutations}. By Lemma \ref{lem:nonzeroHN} we have $a_{111}, b_{111} \neq 0$. This, via Lemma \ref{lem:conv} (iv) means that setting $\lambda_{w} = \phi_{w}/d_{w} - 1$ and $L_{\phi} = \phi - \varepsilon$, considering the convolution semi-group associated to $L_{\phi}$ yields for all $t > 0$
\begin{equation*}
d_{111}e^{\lambda_{111}t} = a_{111} d_{10}e^{\lambda_{10}t} + b_{111}d_{111} e^{\lambda_{1}t}.
\end{equation*}
Therefore $\phi_{10} = (d_{10}/d_{1})\phi_{1}$.

Using this, the same reasoning as in the proof of Theorem \ref{thm:ON+} shows that for any tracial central state,
\begin{equation*}
\phi = \left(1 - \mu\right)h + \frac{1}{2}\left(\mu + \lambda\right)\varepsilon + \frac{1}{2}\left(\mu - \lambda\right)\varepsilon_{\mathrm{alt}}.
\end{equation*}
with $\mu\leqslant 1$. We are therefore left with proving that $\vert \lambda\vert \leqslant\mu$. The first step is to observe that $\phi_{\mid _{\Pol (S_{N}^{+})}} = (1-\mu)h + \mu\varepsilon$, so that $\mu\geqslant 0$. The second one is to consider the characters $\chi_{1^{n}}$. Indeed, we have by parity
\begin{equation*}
\phi(\chi_{1^{n+1}}\chi_{1^{n}}) = \lambda d_{1^{n+1}}d_{1^{n}}
\end{equation*}
while $\phi(\chi_{1^{n}}^{2}) = \mu d_{1^{n}}^{2}$ so that the same argument (involving the Cauchy-Schwarz inequality) as in the end of the proof of Theorem \ref{thm:ON+} yields $\lambda^{2}\leqslant \mu^{2}$. 
\end{proof}

As before, the proof of the last theorem in fact describes also all central linear functionals on $\Pol(H_N^+)$.
\begin{cor}\label{cor:HN}
Let $N \geq 6$.	Any tracial central linear functional on $\Pol(H_{N}^{+})$ is a linear combination of $\varepsilon$, $\varepsilon_{alt}$  and $h$, and therefore extends to a bounded functional on $C^{u}(H_{N}^{+})$.
\end{cor}

\begin{rmk}
One can likely use the same strategy to prove that the only extremal tracial central states on $H_{N}^{s+}$ are either the Haar state or the evaluation at the diagonal matrix with constant coefficient given by an $s$-th root of unity.
\end{rmk}

\renewcommand*{\thesection}{\appendixname}

\appendix
\setcounter{section}{-1}

\section{Computations with the Haar state on quantum permutation groups}
\setcounter{section}{1}

\renewcommand*{\thesection}{\Alph{section}}
\renewcommand*{\thesubsection}{\Alph{section}.\arabic{subsection}}


We have gathered in this appendix some computations involving the Haar state on $S_{N}^{+}$ which were used in Section \ref{sec:permutation}. The reason for that is that these computations rely on techniques which have nothing to do with the main subject of this paper, and are rather lengthy and technical. The authors are thankful to Roland Speicher for discussions on this topic which led to the development of the method.

Before embarking on the details of the proof, let us describe the main tool that we will use. We will follow the notations of Section \ref{sec:permutation} and assume throughout that $N \geq 4$. The problem is that of computing a general moment
\begin{equation*}
 h(p_{i_{1}j_{1}} \cdots p_{i_{k}j_{k}}),
\end{equation*}
for arbitrary $k \in \bn$ and $i_1, \ldots,i_k, j_1, \ldots, j_k \in \{1, \ldots N\}$.
Because $S_{N}^{+}$ contains $S_{N}$, $h$ is invariant under permutations of the first set of indices and permutations of the second set of indices independently. As a consequence, the moment above only depends on which indices in $i = (i_{1}, \cdots, i_{k})$ are equal, and similarly for $j = (j_{1}, \cdots, j_{k})$. In other words, denoting by $\ker(i)$ the partition whose blocks are given by the indices having the same value in $i$, the moment  depends only on the partitions $\pi = \ker(i)$ and $\pi' = \ker(j)$. We will therefore denote such a moment by
\begin{equation*}
h(\pi, \pi').
\end{equation*}

We will not compute explicitly all these numbers, but rather find relations between them. For convenience, we fix some notations:
\begin{align*}
\pi_{4}^{1} & = \{\{1, 3\}, \{2, 4\}\} \\
\pi_{4}^{2} & = \{\{1, 3\}, \{2\}, \{4\}\} \\
\pi_{4}^{2\prime} & = \{\{1\}, \{2, 4\}, \{3\}\} \\
\pi_{4}^{3} & = \{\{1\}, \{2\}, \{3\}, \{4\}\} \\
\pi_{5}^{1} & = \{\{1, 3\}, \{2, 4\}, \{5\}\} \\
\pi_{5}^{2} & = \{\{1, 3\}, \{2, 5\}, \{4\}\} \\
\pi_{5}^{3} & = \{\{1, 3\}, \{2\}, \{4\}, \{5\}\}
\end{align*}

Here is a first example of the kind of relations one can obtain between these values of the Haar state.

\begin{lem}\label{lem:orderfour}
We have
\begin{align*}
h(\pi_{4}^{2}, \pi_{4}^{2}) & = \frac{1}{N(N-1)(N-2)} - \frac{1}{N-2}h(\pi_{4}^{2}, \pi_{4}^{1}), \\
h(\pi_{4}^{2}, \pi_{4}^{1}) & = \frac{1}{N(N-1)(N-2)} - \frac{1}{N-2}h(\pi_{4}^{1}, \pi_{4}^{1}).
\end{align*}
\end{lem}

\begin{proof}
Fix throughout the proof pairwise different indices $i, j, k \in \{1, \ldots, N\}$. We start with the equality
\begin{equation*}
h(p_{i}p_{j}p_{i}) = \sum_{l=1}^{N}h(p_{i}p_{j}p_{i}p_{kl})
\end{equation*}
for some $k\notin\{i, j\}$. In the sum, the term with $l=i$ vanishes since $i\neq k$ (as $p_i p_{ki}=0$ for $i\neq k$), while all terms with $l\notin\{i, j\}$ are equal by permutation invariance. Since $h(p_{i}p_{j}p_{i}) = h(p_{i}p_{j}) = 1/N(N-1)$, we have
\begin{align*}
\frac{1}{N(N-1)} & = h(p_{i}p_{j}p_{i}p_{kj}) + (N-2)h(\pi_{4}^{2}, \pi_{4}^{2}) \\
& = h(\pi_{4}^{2}, \pi_{4}^{1}) + (N-2)h(\pi_{4}^{2}, \pi_{4}^{2}) \\
\end{align*}
which yields the first relation.

As for the second one, we proceed similarly with the equality
\begin{equation*}
h(p_{i}p_{j}p_{i}) = \sum_{l=1}^{N}h(p_{i}p_{j}p_{i}p_{lj}).
\end{equation*} The term for $l = i$ vanishes, yielding
\begin{equation*}
(N-2)h(\pi_{4}^{2}, \pi_{4}^{1}) + h(\pi_{4}^{1}, \pi_{4}^{1}) = h(p_{i}p_{j}p_{i}).
\end{equation*}
\end{proof}

We will also need a similar formula for moments of order $5$.

\begin{lem}\label{lem:orderfive}
We have
\begin{align*}
h(\pi_{5}^{1}, \pi_{5}^{1}) & = \frac{1}{N-2}h(\pi_{4}^{1}, \pi_{4}^{1}), \\
h(\pi_{5}^{2}, \pi_{5}^{2}) & = \frac{1}{N-2}h(\pi_{4}^{1}, \pi_{4}^{1}), \\
h(\pi_{5}^{3}, \pi_{5}^{3}) & = \frac{1}{N-3}h(\pi_{4}^{2}, \pi_{4}^{2}) - \frac{1}{(N-2)(N-3)}h(\pi_{4}^{2}, \pi_{4}^{1}). \\
\end{align*}
\end{lem}

\begin{proof}
	Fix throughout the proof pairwise different indices $i, j, k \in \{1, \ldots, N\}$.
First observe that by traciality,
\begin{align*}
h(\pi_{5}^{1}, \pi_{5}^{1}) & = h(p_{i}p_{j}p_{i}p_{j}p_{k}) = h(p_{j}p_{i}p_{j}p_{k}p_{i}) = h(\pi_{5}^{2}, \pi_{5}^{2}).
\end{align*}
Therefore, the first two moments in the statement are equal. Now
\begin{align*}
h(\pi_{4}^{1}, \pi_{4}^{1}) & = h(p_{i}p_{j}p_{i}p_{j}) = \sum_{l=1}^{N}h(p_{i}p_{j}p_{i}p_{j}p_{kl}).
\end{align*}
In the sum above, the terms for $l = i$ and $l = j$ vanish and the $(N-2)$ other ones equal $h(\pi_{5}^{1}, \pi_{5}^{1})$, concluding the proof.

As for the last equality, we start with
\begin{equation*}
h(p_{i}p_{j}p_{i}p_{k}) = \sum_{m=1}^{N}h(p_{i}p_{j}p_{i}p_{k}p_{lm})
\end{equation*}
for some $l \in \{1, \ldots,N\}$, $l\notin\{i, j, k\}$. The terms for $m = i$ and $m = k$ vanish, so that we get
\begin{equation*}
(N-3)h(\pi_{5}^{3}, \pi_{5}^{3}) + h(\pi_{5}^{3}, \pi_{5}^{2}) = h(\pi_{4}^{2}, \pi_{4}^{2}).
\end{equation*}
Considering again some $l\notin\{i, j, k\}$ and
\begin{equation*}
h(p_{i}p_{j}p_{i}p_{lj}) = \sum_{n=1}^{N}h(p_{i}p_{j}p_{i}p_{kn}p_{lj}),
\end{equation*}
we see that the terms for $n = i$ and $n = j$ vanish and all the other ones are equal. Thus
\begin{equation*}
(N-2)h(\pi_{5}^{3}, \pi_{5}^{2}) = h(\pi_{4}^{2}, \pi_{4}^{1})
\end{equation*}
and the result follows.
\end{proof}

We can now use all this to check our condition on the coefficient given by \eqref{b3}. For the sake of clarity, we will first restate the condition we need in terms of moments involving only $\pi_{4}^{1}$.

\begin{prop}\label{prop:equivalentnonvanishing}
For $i, j \in \{1, \ldots, N\}$, $i\neq j$, we have
\begin{equation*}
h(p_{i}p_{j}p_{i}\chi_{2}) = \frac{1}{N(N-1)}
\end{equation*}
if and only if
\begin{equation*}
h(\pi_{4}^{1}, \pi_{4}^{1}) = \frac{1}{2(N-1)^{2}}.
\end{equation*}
\end{prop}

\begin{proof}
The starting point is the decomposition
\begin{equation*}
\chi_{2} = \chi^{2} - 3\chi + 1
\end{equation*}
which enables us to split the computation into three terms. The one involving the constant $1$ is straightforward to compute. As for the one involving $\chi$, we have
\begin{align*}
h(p_{i}p_{j}p_{i}\chi) & = h(p_{i}p_{j}p_{i}) + h(p_{i}p_{j}p_{i}p_{j}) + \sum_{k\notin\{i, j\}}h(p_{i}p_{j}p_{i}p_{k}). \\
\end{align*}
We are now going to compute the sum
\begin{equation*}
h(p_{i}p_{j}p_{i}\chi^{2}) = \sum_{k, l = 1}^{N}h(p_{i}p_{j}p_{i}p_{k}p_{l}).
\end{equation*}
Let us first consider the case where $k = l$. There are two special values, namely $k = l = i$ and $k = l = j$, and then $N-2$ other terms which are all equal. If instead $k\neq l$, then we have several possible cases again:
\begin{itemize}
\item $l = i$ so that we get $h(p_{i}p_{j}p_{i}p_{k})$;
\item $k = i$ so that we get $h(p_{i}p_{j}p_{i}p_{l})$;
\item $l = j$ and $k\neq i$, or $k = j$ and $l \neq i$. 
\end{itemize}
The remaining terms correspond to the case where $k\neq l$ and none of them belongs to $\{i, j\}$. In other words,
\begin{align*}
h(p_{i}p_{j}p_{i}\chi^{2}) & = h(p_{i}p_{j}p_{i}) + h(p_{i}p_{j}p_{i}p_{j}) + \sum_{k\notin\{i, j\}}h(p_{i}p_{j}p_{i}p_{k}) \\
& \;\;\;\; + \sum_{k\neq i}h(p_{i}p_{j}p_{i}p_{k}) + \sum_{l\neq i}h(p_{i}p_{j}p_{i}p_{l}) + \sum_{k\notin \{i, j\}}h(p_{i}p_{j}p_{i}p_{k}p_{j}) + \sum_{l\notin\{i, j\}}h(p_{i}p_{j}p_{i}p_{j}p_{l})\\
& \;\;\;\; + \sum_{k \neq l\notin \{i, j\}}h(p_{i}p_{j}p_{i}p_{k}p_{l}) \\
& = h(p_{i}p_{j}p_{i}) + 3h(p_{i}p_{j}p_{i}p_{j}) + 3\sum_{k\notin\{i, j\}}h(p_{i}p_{j}p_{i}p_{k}) \\
& \;\;\; \;+ \sum_{k\notin \{i, j\}}h(p_{i}p_{j}p_{i}p_{k}p_{j}) + \sum_{l\notin\{i, j\}}h(p_{i}p_{j}p_{i}p_{j}p_{l}) + \sum_{k \neq l\notin \{i, j\}}h(p_{i}p_{j}p_{i}p_{k}p_{l}). \\
\end{align*}
Gathering everything we eventually get
\begin{align*}
h(p_{i}p_{j}p_{i}\chi_{2}) & = h(p_{i}p_{j}p_{i}) + 3h(p_{i}p_{j}p_{i}p_{j}) + 3\sum_{k\notin\{i, j\}}h(p_{i}p_{j}p_{i}p_{k}) + \sum_{k \neq l\notin \{i, j\}}h(p_{i}p_{j}p_{i}p_{k}p_{l}) \\
& \;\;\;\; + \sum_{k\notin \{i, j\}}h(p_{i}p_{j}p_{i}p_{k}p_{j}) + \sum_{l\notin\{i, j\}}h(p_{i}p_{j}p_{i}p_{j}p_{l}) \\
& \;\;\;\; - 3\left(h(p_{i}p_{j}p_{i}) + h(p_{i}p_{j}p_{i}p_{j}) + \sum_{k\notin\{i, j\}}h(p_{i}p_{j}p_{i}p_{k})\right) \\
& \;\;\;\; + h(p_{i}p_{j}p_{i}) \\
& = \sum_{k \neq l\notin \{i, j\}}h(p_{i}p_{j}p_{i}p_{k}p_{l}) + \sum_{k\notin \{i, j\}}h(p_{i}p_{j}p_{i}p_{k}p_{j}) + \sum_{l\notin\{i, j\}}h(p_{i}p_{j}p_{i}p_{j}p_{l}) - h(p_{i}p_{j}p_{i})\\
& = (N-2)(N-3)h(\pi_{5}^{3}, \pi_{5}^{3}) + (N-2)h(\pi_{5}^{1}, \pi_{5}^{1}) + (N-2)h(\pi_{5}^{2}, \pi_{5}^{2}) - h(p_{i}p_{j}p_{i}) \\
& = (N-2)h(\pi_{4}^{2}, \pi_{4}^{2}) - h(\pi_{4}^{2}, \pi_{4}^{1}) + 2(N-2)\frac{1}{N-2}h(\pi_{4}^{1}, \pi_{4}^{1}) - \frac{1}{N(N-1)} \\
& = \frac{1}{N(N-1)} - 2h(\pi_{4}^{2}, \pi_{4}^{1}) + 2(N-2)\frac{1}{N-2}h(\pi_{4}^{1}, \pi_{4}^{1}) - \frac{1}{N(N-1)} \\
& = 2\left(h(\pi_{4}^{1}, \pi_{4}^{1}) - h(\pi_{4}^{2}, \pi_{4}^{1})\right) \\
& = 2\frac{N-1}{N-2}h(\pi_{4}^{1}, \pi_{4}^{1}) - \frac{2}{N(N-1)(N-2)}.
\end{align*}
If this was equal to $1/N(N-1)$, then we would have
\begin{align*}
h(\pi_{4}^{1}, \pi_{4}^{1}) & = \frac{N-2}{2(N-1)}\left(\frac{1}{N(N-1)} + \frac{2}{N(N-1)(N-2)}\right) \\
& = \frac{1}{2(N-1)^{2}}.
\end{align*}
\end{proof}

The previous criterion will be the key to the main result of this appendix. Indeed, we claim that $h(p_{i}p_{j}\chi_{2}) = 1/N(N-1)$. This follows from the following calculation (recall that $\chi = \chi_{1} + \chi_{0})$:
\begin{align*}
h(p_{i}p_{j}\chi_{2}) & = \frac{1}{N-1}\left(h(p_{i}\chi\chi_{2}) - h(p_{i}\chi_{2})\right) \\
& = \frac{1}{N(N-1)}h(\chi^{2}\chi_{2}) - \frac{1}{N(N-1)}h(\chi\chi_{2}) \\
& = \frac{1}{N(N-1)}.
\end{align*}

Gathering things together now leads to what we need.

\begin{thm}\label{thm:computationb3}
Let $N\geqslant 4$, and let $i,j\in \{1, \ldots,N\}$, $i \neq j$. Then 	$h(p_{i}p_{j}\chi_{2}) \neq h(p_{i}p_{j}p_{i}\chi_{2})$.
\end{thm}

\begin{proof}
The proof relies on the invariance of the Haar state with respect to the coproduct, which can be expressed as
\begin{equation*}
h(u_{i_{1}j_{1}}u_{i_{2}j_{2}}u_{i_{3}j_{3}}u_{i_{4}j_{4}}) = \sum_{k_{1}, k_{2}, k_{3}, k_{4} = 1}^{N}h(u_{i_{1}k_{1}}u_{i_{2}k_{2}}u_{i_{3}k_{3}}u_{i_{4}k_{4}})h(u_{k_{1}j_{1}}u_{k_{2}j_{2}}u_{k_{3}j_{3}}u_{k_{4}j_{4}}).
\end{equation*}
Writing it in terms of partitions then yields
\begin{equation*}
h(\pi_{4}^{1}, \pi_{4}^{1}) = \sum_{\pi \in \mathcal{P}(4)}\frac{N!}{(N-b(\pi))!}h(\pi_{4}^{1}, \pi)h(\pi, \pi_{4}^{1})
\end{equation*}
(here $\mathcal{P}(4)$ denotes the set of all partitions of $\{1,2,3,4\}$ and $b(\pi)$ the number of blocks of a partition $\pi$). In the sum above, many terms vanish for trivial reasons. Indeed, if there are two neighbouring points which are connected in $\pi$ but not in $\pi_{4}^{1}$, then $h(\pi_{4}^{1}, \pi) = 0$ because there are two distinct terms in the same column which are multiplied. A quick inspection of the possibilities then shows that the only non-zero terms correspond to the following partitions : $\pi_{4}^{1}$, $\pi_{4}^{2}$, $\pi_{4}^{2\prime}$ and $\pi_{4}^{3}$. 

Moreover, the Haar state is invariant (in a natural sense) under taking adjoints and under the antipode. Composing both operations and remembering that by the Weingarten formula all the coefficients above are real, shows that it is invariant under exchanging the left and right indices. In other words, $h(\pi_{4}^{1}, \pi) = h(\pi, \pi_{4}^{1})$. 
Thus,
\begin{align*}
h(\pi_{4}^{1}, \pi_{4}^{1}) & = N(N-1)h(\pi_{4}^{1}, \pi_{4}^{1})^{2} + 2N(N-1)(N-2)h(\pi_{4}^{1}, \pi_{4}^{2})^{2} \\
& + N(N-1)(N-2)(N-3)h(\pi_{4}^{1}, \pi_{4}^{3})^{2},
\end{align*}
where we have used the fact that by traciality, $h(\pi_{4}^{1}, \pi_{4}^{2}) = h(\pi_{4}^{1}, \pi_{4}^{2\prime})$. We already have an expression of $h(\pi_{4}^{1}, \pi_{4}^{2})$ in terms of $h(\pi_{4}^{1}, \pi_{4}^{1})$ from Lemma \ref{lem:orderfour}, and we will now provide a similar one for the last term. Let  $k, l, m \in \{1, \ldots, N\}$ be all distinct. We then have 
\begin{equation*}
0 = h(p_{ik}p_{jl}p_{im}) = \sum_{n=1}^N h(p_{ik}p_{jl}p_{im}p_{jn}).
\end{equation*}
The sum over $n$ contains two vanishing terms (corresponding to $k$ and $m$), one term for $n = l$ and $N-3$ other terms:
\begin{equation*}
h(p_{ik}p_{jl}p_{im}) = h(\pi_{4}^{1}, \pi_{4}^{2}) + (N-3)h(\pi_{4}^{1}, \pi_{4}^{3}).
\end{equation*}
We conclude that
\begin{align*}
h(\pi_{4}^{1}, \pi_{4}^{3}) & = - \frac{1}{N-3}h(\pi_{4}^{1}, \pi_{4}^{2}) = - \frac{1}{N-3}\left(\frac{1}{N(N-1)(N-2)} - \frac{1}{N-2}h(\pi_{4}^{1}, \pi_{4}^{1})\right).
\end{align*}
Setting $X = h(\pi_{4}^{1}, \pi_{4}^{1})$, we end up with the following quadratic equation:
\begin{align*}
X & = N(N-1)X^{2} + 2N(N-1)(N-2)\left(\frac{1}{N(N-1)(N-2)} - \frac{X}{N-2}\right)^{2} \\
& + N(N-1)(N-2)(N-3)\frac{1}{(N-3)^{2}}\left(\frac{1}{N(N-1)(N-2)} - \frac{X}{N-2}\right)^{2}.
\end{align*}
Setting $\widetilde{X} = N(N-1)X$, the equation becomes
\begin{equation*}
\widetilde{X} = \widetilde{X}^{2} + \frac{2N-5}{(N-2)(N-3)}(1-\widetilde{X})^{2}.
\end{equation*}
Setting $\alpha_{N} = \frac{2N-5}{(N-2)(N-3)}$, the equation becomes
\begin{equation*}
(1+\alpha)\widetilde{X}^{2} - (1+2\alpha)\widetilde{X} + \alpha = 0\,\, ,
\end{equation*}
whose solutions are $(1+2\alpha\pm 1)/(2+2\alpha)$. One of these is the obvious solution $\widetilde{X} = 1$, while the other one is
\begin{equation*}
\widetilde{X} = \frac{2N-5}{N^{2} - 3N + 1}.
\end{equation*}
None of these equals $\frac{N}{2(N-1)}$, and the result therefore follows from Proposition \ref{prop:equivalentnonvanishing}.
\end{proof}

\begin{rmk}
Once a first version of this work was made available, we were informed that the value of $h(\pi_{4}^{1}, \pi_{4}^{1})$ was already computed in \cite[Thm 4.4]{McC} (see \cite{McC2} for a misprint in the displayed formula). This provides an alternative proof of Theorem \ref{thm:computationb3}.
\end{rmk}

 \noindent {\bf Acknowledgments. }
 A.S.\ was  partially supported by the National Science Center (NCN) grant no. 2020/39/I/ST1/01566, and would also like to gratefully acknowledge the visit to Harbin in autumn of 2023 when the work on this paper was initiated. S.W.\ was  partially supported by the NSF of China (No.12031004, No.12301161)
 and the Fundamental Research Funds for the Central Universities. We thank Jacek Krajczok for useful comments. Finally we thank the referee for a very thorough reading of our paper and numerous suggestions significantly improving the presentation.

\end{document}